\documentclass[11pt,leqno]{amsart}
\usepackage{amsfonts, amsmath, amsthm, amssymb,xspace}
\usepackage[breaklinks=true]{hyperref}

\theoremstyle{plain}
\newtheorem{theorem}{Theorem}[section]
\newtheorem{lemma}{Lemma}[section]
\newtheorem{prop}{Proposition}[section]
\newtheorem{cor}{Corollary}[section]

\theoremstyle{definition}
\newtheorem{defin}{Definition}[section]
\newtheorem{remark}{Remark}[section]
\newtheorem*{acknowledgement}{Acknowledgements}

\newcommand{\half}{\displaystyle{\frac{1}{2}}}

\newcommand{\Z}{\mathbb Z}
\newcommand{\Q}{\mathbb Q}
\newcommand{\h}{{\bf h}}
\newcommand{\bggo}{\mathcal O}
\newcommand{\cala}{\mathcal A}
\newcommand{\mf}[1]{\displaystyle{\mathfrak{#1}}}
\newcommand{\g}{\displaystyle{\widetilde{\mathfrak{g}}}}
\newcommand{\mfn}{\displaystyle{\mathfrak{n}}}
\newcommand{\mfh}{\displaystyle{\mathfrak{h}}}

\newcommand{\mapdef}[1]{\ensuremath{\overset{#1}{\longrightarrow}}\xspace}
\newcommand{\comment}[1]{}

\DeclareMathOperator{\tr}{\ensuremath{Tr}}
\DeclareMathOperator{\spec}{\ensuremath{Spec}}
\DeclareMathOperator{\Gr}{\ensuremath{gr}}
\DeclareMathOperator{\Ch}{\ensuremath{ch}}
\DeclareMathOperator{\ad}{\ensuremath{ad}}
\DeclareMathOperator{\Id}{\ensuremath{Id}}
\DeclareMathOperator{\im}{\ensuremath{im}}
\DeclareMathOperator{\Sym}{\ensuremath{Sym}}
\DeclareMathOperator{\Ann}{\ensuremath{Ann}}
\DeclareMathOperator{\End}{\ensuremath{End}}

\DeclareMathOperator{\Hom}{\ensuremath{Hom}}

\begin{document}
\thanks{This paper is essentially an improved write-up of A.T.'s
Ph.D.~thesis from the University of Chicago, and Section \ref{S2} is a
more detailed exposition of an earlier preprint \cite{T}.}

\title[Center and representations of infinitesimal Hecke algebras]{Center
and representations of infinitesimal Hecke algebras of $\mf{sl}_2$}


\author{Akaki Tikaradze}
\address{Department of Mathematics\\
The University of Toledo\\
Toledo, OH - 43606, USA}
\email[A.~Tikaradze]{\tt tikar@math.uchicago.edu}

\author{Apoorva Khare}
\address{Department of Mathematics\\
Yale University\\
New Haven, CT - 06511, USA}
\email[A.~Khare]{\tt apoorva.khare@yale.edu}

\begin{abstract}
In this paper, we compute the center of the infinitesimal Hecke algebras
$H_z$ associated to $\mf{sl}_2$; then using nontriviality of the center,
we study representations of these algebras in the framework of the BGG
category $\bggo$. We also discuss central elements in infinitesimal Hecke
algebras over $\mf{gl}_n$ and $\mf{sp}(2n)$ for all $n$. We end by
proving an analogue of Duflo's theorem for $H_z$.
\end{abstract}
\maketitle

\section{Introduction}

\subsection{Background}
In the paper \cite{EGG}, the authors introduce new families of algebras
which they call continuous Hecke algebras and infinitesimal Hecke
algebras (the latter being subalgebras of the former). They do this as a
way to provide a unifying treatment of the representation theories of
various algebras such as Drinfeld-Lusztig degenerate affine Hecke
algebras, and symplectic reflection algebras of \cite{EG} (which include
rational Cherednik algebras). We briefly recall their definition.
 
We fix once and for all a ground field $k$ (which will be assumed to be
algebraically closed of characteristic zero), and let $G$ be a reductive
algebraic group over $k$ (not necessarily connected), and $\rho :G\to
GL(V)$ a finite-dimensional representation. Then one can form the
semi-direct product algebra $TV \rtimes \bggo(G)^*$, where $TV$ is the
tensor algebra of $V$ and $\bggo(G)^*$ is the algebra of algebraic
distributions on $G$.

Now given a skew-symmetric $G$-equivariant $k$-linear pairing $\gamma :V
\times V \to \bggo(G)^*$, the authors define in \cite{EGG} an algebra
$H_{\gamma}(G)$, as a quotient of $TV \rtimes \bggo(G)^*$ by the
relations: $[x,y] = \gamma(x,y)$ for all $x,y \in V$.

One has an algebra filtration on $H_{\gamma}(G)$ obtained by assigning to
$V$ the filtration degree 1, and 0 to $\bggo(G)^{*}$. Hence we get a
natural map $: H_0(G) \twoheadrightarrow \Gr(H_{\gamma}(G))$, and
$H_{\gamma}(G)$ is called a {\it continuous Hecke algebra} if and only if
this map is an isomorphism (the {\it PBW property}).

If one takes distributions supported on $1\in G$, instead of
$\bggo(G)^{*},$ the resulting algebra is called an {\it infinitesimal
Hecke algebra} if the corresponding PBW property is satisfied. Hence this
algebra is a quotient of $TV \rtimes \mf{Ug}$ by a $\mf{g}$-invariant
relation: $[x,y]=\gamma(x,y)$, where $\gamma: V \times V \to \mf{Ug}$.
It is also a deformation of $\mf{Ug} \ltimes \Sym(V) = \mf{U}(\mf{g}
\ltimes V)$.

If $G$ is connected, one gets a continuous Hecke algebra if and only if
the corresponding algebra is an infinitesimal Hecke algebra. When $G$ is
a discrete group, one recovers the symplectic reflection algebras of
\cite{EG} in this way. So in a sense, symplectic reflection algebras and
infinitesimal Hecke algebras lie on opposite sides of the
spectrum.\medskip

In this paper, we will mainly be concerned with the question of computing
the center of the infinitesimal Hecke algebras of $SL_2$, and the
spectral decomposition for the analogue of the BGG category $\bggo$ for
these, over the center. It is well-known (\cite{BG}) that the center of
symplectic reflection algebras is either trivial, or the whole algebra is
a finitely generated module over its center (when the one-dimensional
parameter is 0).

It seems to us that one has a completely opposite picture for
infinitesimal Hecke algebras. Namely, infinitesimal Hecke algebras of
$SL_2$ and $GL_2$ have nontrivial (but not ``large") centers, so the
category $\bggo$ has a spectral decomposition. We expect similar
phenomena for infinitesimal Hecke algebras of higher rank as
well.\medskip

\subsection{Results}
We now describe (some of) the concrete results of the paper.
   
For the most part, we will work with $\mf{g} = \mf{sl}_2$ and $V=k^2$,
the standard representation with basis vectors $x, y$. In this case we
have $H_z = (TV \rtimes \mf{Ug}) / ([x,y]-z)$, where $z$ is a central
element of $\mf{Ug}$.

\begin{itemize}
\item We prove (Theorem \ref{T1}) that the center of $H_z$ is freely
generated by a nontrivial quadratic element for any value of $z$
(quadratic with respect to the filtration that assigns degree 1 to $V$
and 0 to $\mf{g}$). This central element also exists for $\mf{g} =
\mf{sp}(2n)$ and $V = k^{2n}$, at least when the deformation parameter is
trivial.

\item Moreover, it is shown (also in Theorem \ref{T1}) that this algebra
has no outer derivations for nonzero $z$, and if $z=0$, then the Euler
derivation generates the outer derivations.

\item The commutator quotient of $H_z$ turns out to be finitely generated
over the center (Theorem \ref{Tzz}); it is generated by $\deg(z)$
elements (where we look at $z$ as a polynomial in the Casimir element).
\end{itemize}\medskip

We also briefly consider the infinitesimal Hecke algebra associated with
$\mf{g} = \mf{gl}_n$ and $V = \mf{h} \oplus \mf{h}^{*}$, where
$\mf{h}=k^n$ is the standard representation. In this case (at least when
$\beta \equiv 0$), the center of $H_\beta$ contains at least two
(algebraically independent) quadratic elements. Moreover, we prove that
for any $\beta$, the center of $H_\beta$ is nontrivial (see Proposition
\ref{Pgln}).\medskip

We then consider some consequences of the nontriviality of the center of
$H_z$, such as the spectral decomposition of the BGG category $\bggo$,
the Harish-Chandra homomorphism, and so on. We also describe the
multiplicities of irreducible modules in Verma modules when the parameter
is a scalar.

Finally, we prove an analogue of Duflo's theorem on primitive ideals for
the infinitesimal Hecke algebra $H_z$, by utilizing a theorem of Ginzburg
\cite{Gi}.\medskip

\section{The center}\label{S2}

Let us start by recalling the exact definition of infinitesimal Hecke
algebras for $\mf{g} = \mf{sp}(2n)$ and $V=k^{2n}$. Denote by $\omega$
the symplectic form on $V$; one then identifies $\mf{g}$ with $\mf{g}^*$
via the pairing $\mf{g} \times \mf{g} \to k,\ (A,B) \mapsto \tr(AB)$, and
$\Sym \mf{g}$ with $\mf{Ug}$ via the symmetrization map. Then for any
$x,y \in V,\ A \in \mf{g}$, one writes
\[ \omega(x, (1 - T^2 A^2)^{-1} y) \det(1 - T A)^{-1} = l_0(x,y)(A) +
l_2(x,y)(A) T^2 + \dots \]

\noindent where $l_i(x,y) \in \Sym \mf{g} \cong \mf{Ug}$ is a polynomial
in $\mf{g}$ for each $i$.

For each polynomial $\beta = \beta_0 + \beta_2 T^2 + \beta_4 T^4 + \dots
\in k[T]$, in \cite{EGG} the authors define the algebra $H_{\beta}$ to be
the quotient of $TV \rtimes \mf{Ug}$ by the relations
\[ [x,y]=\beta _0 l_0(x,y) + \beta _2 l_2(x,y) + \dots \]

\noindent for all $x,y \in V$. It is proved in \cite{EGG} that this
yields an infinitesimal Hecke algebra (i.e., the PBW property
holds). Also note that setting $\beta \equiv 0$ yields the ``undeformed"
case: $H_0(\mf{sp}(2n)) = \mf{U}(\mf{sp}(2n) \ltimes k^{2n})$.\medskip

We will restrict ourselves to the case $n=1$. Let us describe more
explicitly a presentation (via generators and relations) of this algebra
(e.g., see \cite[Example 4.12]{EGG}).
We have $V = kx \oplus ky$, with $[h,x] = x, [h,y] = -y$ (where $e,f,h$
form the standard basis for $\mf{sl}_2$, with standard relations $[h, e]
= 2e, [h, f] = -2f$, and $[e, f] = h$). Then this algebra is a quotient
of $TV \rtimes \mf{Ug}$ by the relation $[x,y] = z$, where $z$ is a
central element of $\mf{Ug}$. We will denote this algebra by $H_z$.

A few years before the paper \cite{EGG} appeared, the representation
theory of $H_z$ was studied in great detail by A.~Khare in \cite{Kh}. In
particular, he proved the PBW property there, the proof being completely
different from the one in \cite{EGG}.\medskip

We start by determining the center and derivations of the algebra $H_z$.
We have the following

\begin{theorem}\label{T1}\hfill
\begin{enumerate}
\item The center of $H_z$ is a polynomial algebra in one variable, and
the generating central element has filtration degree $2$.

\item If $z = 0$, then $H^1(H_0, H_0)$ (Hochschild cohomology) is a rank
one free module over the center, and if $z \neq 0$, then every derivation
of $H_z$ is inner.
\end{enumerate}
\end{theorem}

\noindent We prove the theorem in several steps, showing several small
results along the way. It is noteworthy that if we replace $H_z$ by its
natural quantization, then (if $z \neq 0$) the center becomes trivial;
see \cite[Theorem 11.1]{GK}.\medskip

\subsection{An anti-involution and a central element}

First, recall an (algebra) anti-isomorphism of $H_z$, called $j$, defined
in \cite{Kh}:
\[ j(x)=y, \quad j(y)=x, \quad j(h)=h, \quad j(e)=-f, \quad j(f)=-e. \]

\noindent More generally, let us also write down a basis for
$\mf{sp}(2n)$:
\begin{equation}
u_{jk} := e_{jk} - e_{k+n,j+n}, \quad v_{jk} := e_{j,k+n} + e_{k,j+n},
\quad w_{jk} := e_{j+n,k} + e_{k+n,j}
\end{equation}

We now claim

\begin{lemma}\label{L1}
Let $\Delta=h^2+4ef-2h$ be a multiple of the Casimir element of
$\mf{sl}_2$.
\begin{enumerate}
\item The map $j$, taking $u_{jk} \leftrightarrow u_{kj}, v_{jk}
\leftrightarrow -w_{jk}$, and $e_i \leftrightarrow e_{i+n}$ (in $V =
k^{2n}$) for all $1 \leq i \leq n$, is an anti-involution of
$\mf{U}(\mf{sp}(2n) \ltimes TV)$.

\item It also factors through an anti-involution of
$H_\beta(\mf{sp}(2n))$ for scalar parameters $\beta_0$, as well as for
all $H_z$ (here, $n=1$ and $z$ is any central element in $\mf{Ug}$).

\item For $n=1$ and any $z$, the map $j$ fixes the following elements in
$H_z:\ h,\ \Delta,\ z,\ t := ey^2+hxy-fx^2$.

\item Moreover, the element $t - \half hz$ commutes with $e,f,h$ in
$H_z$.
\end{enumerate}
\end{lemma}

\begin{proof}\hfill
\begin{enumerate}
\item Consider $\mf{sp}(2n) \hookrightarrow \mf{gl}(2n)$. Then on
$\mf{sp}(2n)$, $j$ is the map $j(X) := - \tau X \tau^{-1}$, where
$\displaystyle \tau = \tau^{-1} = \begin{pmatrix} 0 & \Id_n\\ \Id_n &
0\end{pmatrix}$. On $V$, $j$ is the map $v \mapsto \tau \cdot v$. One now
easily checks that this yields an anti-involution of $\mf{Ug} \ltimes
TV$.

\item For a scalar parameter $\beta_0$, the added relations we have to
quotient $\mf{U}(\mf{sp}(2n)) \ltimes TV$ by, are: $[e_i, e_k] = \beta_0
\delta_{|i-k|, n} (i-k)/n$. These are clearly preserved by $j$.
Similarly, $j$ preserves $[x,y]$ as well as $z = z(\Delta)$.

\item That $j$ fixes $h$ and $\Delta$ (and hence $z$) is easy. Now
applying $j$ to $t$, we get
\[ j(t) = -x^2f+xyh+y^2e=hxy+ey^2-fx^2-[e,y^2]-(-[f,x^2]). \]

\noindent But the last two terms cancel each other, since
\[ [e,y^2] = [e,y] y + y[e,y] = xy + yx = x[f,x] + [f,x]x = [f,x^2], \]

\noindent so this element is indeed fixed by $j$.

\item Note that
\[ [e,t] = e(xy+yx) - 2exy + hx^2 - hx^2 = eyx-exy = -ez, \]

\noindent so we see that $[e,t - \half hz]=0$. Moreover, $t - \half hz$
also commutes with $h$. Finally, applying $-j$ to $et = te$, we get $tf =
ft$.
\end{enumerate}
\end{proof}

Though we do not use it in this manuscript, we now generalize the above
central element (note that $t \in \mf{Z}(H_0)$) for all $n$:

\begin{prop}\label{Psp}
For any $n$, the ``undeformed" algebra $H_0(\mf{sp}(2n))$ has at least
one central element, namely:
\[ t_n := \sum_{1 \leq r,s \leq n} (v_{rs} e_{r+n} e_{s+n} + u_{rs} e_s
e_{r+n} + u_{sr} e_r e_{s+n} - w_{rs} e_r e_s) \]

\noindent where $\{ e_i : 1 \leq i \leq 2n \}$ is the standard basis of
$V = k^{2n}$.
\end{prop}

\noindent Note that if $n=1$, then $t_n = 2t$.

\begin{proof}
We outline the steps of this long-winded but straightforward (and heavily
computational) proof. Define $a_{rs} := v_{rs} e_{r+n} e_{s+n} - w_{rs}
e_r e_s$, and $b_{rs} := u_{rs} e_s e_{r+n} + u_{s,r} e_r e_{s+n}$ for
all $r,s$. The steps of the verification are:
\begin{enumerate}
\item The anti-involution $j$ (in Lemma \ref{L1}) preserves $a_{rs},
b_{rs}$ for all $r,s$; hence it preserves $t_n$ too.

\item $[e_i, a_{rs} + b_{rs}] = 0$ for all $r,s$ and $1 \leq i \leq n$;
hence the same holds by replacing $e_i$ by $e_{i+n}$, using $j$.

\item $\left[ u_{pq}, \sum_{r,s = 1}^n a_{rs} \right] = \left[ u_{pq},
\sum_{r,s = 1}^n b_{rs} \right] = 0\ \forall p,q$.

\item $[v_{pq}, t_n] = 0\ \forall p,q$, whence $[w_{pq}, t_n] = 0$ using
$j$.
\end{enumerate}
\end{proof}

\subsection{Commutators of powers of the Casimir element}

By Lemma \ref{L1}, $j$ fixes the subalgebra generated by the elements $t,
h$, and $\mf{Z}(\mf{Ug})$ (the center of $\mf{Ug}$).
Hence our goal now is to exhibit an element from this algebra which will
commute with $e,x,h$ (and hence with $y,f$, applying $j$), and therefore
will lie in the center of $H_z$.

We now compute that
\begin{eqnarray*}
[x,t] & = & e(zy+yz)-x^2y+hxz+yx^2\\
& = & 2ezy-e[z,y]+hzx-h[z,x]-(2zx-[z,x])\\
& = & 2ezy-2zx+[z,x]-e[z,y]+hzx-h[z,x],
\end{eqnarray*}
and
\[ [x,\half hz] = -\half xz + \half h[x,z] = -\half zx + \half [z,x] +
\half h[x,z], \]

\noindent so we get that
\begin{equation}\label{Ereqd}
[x,t-\half hz] = 2ezy - \frac{3}{2} zx + \half[z,x] - e[z,y] - \half
h[z,x] + hzx.
\end{equation}

\noindent Denote this element by $\omega$. We now want to produce an
element $q_z$ in the center of $\mf{Ug}$ such that $[x, q_z] = \omega$,
for then $t - \half hz - q_z$ will be a central element in $H_z$.

To show this, we will analyze $\mf{sl}_2$-maximal vectors in $\mf{Ug}$
(i.e., vectors annihilated by the adjoint action of $e$) and in $H_z$, of
various weights. A first step in looking at such things is realizing that
$H_z$ is a direct sum of finite-dimensional $\mf{g}$-modules (this is
true for any infinitesimal Hecke algebra):

\begin{lemma}\label{L0}
Given Lie algebras $\mf{g} \neq 0,\mf{h}'$ that are semisimple and
abelian respectively, define $\mf{h} := \mf{h}_{\mf{g}} \oplus \mf{h}'$,
the Cartan subalgebra of the reductive Lie algebra $\g := \mf{g} \oplus
\mf{h}'$.
If $V$ is an $\mf{h}$-semisimple completely reducible $\g$-module, then
so is $\cala := \mf{U}\g \ltimes TV$.
\end{lemma}

\begin{cor}\label{C0}
Every infinitesimal Hecke algebra is such a direct sum, and of
finite-dimensional $\g$-modules.
\end{cor}

The corollary is obvious since such algebras are quotients of $\cala$ for
some finite-dimensional $V$ (so that all ``highest weights" of summands
in $\cala$ are sums of two dominant integral weights for $\mf{g}$, one
from each tensor factor $\mf{U}\g, TV$).

\begin{proof}[Proof of Lemma \ref{L0}]
The $\mf{h}$-semisimplicity is obvious. It is also easy to check that
$\cala$ is graded: $\cala = \bigoplus_{n,I} \cala_{n,I}$. Here,
$\cala_{n,I} := \mf{Ug} \otimes (k \cdot I) \otimes T^n V$, where $n \geq
0$ and $I$ runs over some fixed basis of $\Sym \mf{h}'$. Moreover, each
summand has an increasing filtration by finite-dimensional $\g$-modules,
using the standard filtration on $\mf{Ug}$:
\[ \cala_{n,I} = \cala_{n,I}^\bullet = (F^\bullet\mf{Ug}) \otimes (k
\cdot I) \otimes T^n V. \]

\noindent Using Zorn's lemma, one easily shows that a union of
finite-dimensional (and hence completely reducible) $\mf{h}$-semisimple
$\g$-modules is itself completely reducible. But then, so is $\cala =
\bigoplus_{n,I} \cala_{n,I}$.
\end{proof}

Next, we have

\begin{lemma}\label{L2}
The map $\varphi : k[X,Y] \to \mf{Ug}$, sending $X^m Y^n \mapsto \Delta^m
e^n$, is a vector space isomorphism onto the set of maximal vectors in
(the $\ad \mf{g}$-module) $\mf{Ug}$.
\end{lemma}

\begin{proof}
The injectivity is obvious. Now let $\alpha$ be such a maximal vector. We
may assume without loss of generality, that $\alpha$ is in one weight
space. We proceed by induction on the weight. If $\alpha$ is divisible by
$e$ and $\alpha = ge$ for some $g \in \mf{Ug}$, then we claim that $[g,e]
= 0$ too. For we have
\begin{equation}\label{Etrick}
0 = [\alpha,e] = [ge,e] = [g,e]e \Rightarrow [g,e] = 0
\end{equation}

\noindent since $\mf{Ug}$ is an integral domain. (We will use this {\bf
dividing trick} later in this manuscript.)

Thus, we now assume that $\alpha$ is not divisible by $e$, so if we write
it in the usual PBW basis, it will contain a monomial $a$ containing no
$e$.
Thus $a$ has non-positive weight, and since (by Lemma \ref{L0}, with
$\mf{h}' = V = 0$) $\mf{Ug}$ is a direct sum of finite dimensional
$\mf{g}$-modules (under the adjoint action), it has no maximal vectors of
negative weight. Therefore $a$ has weight 0, and hence is annihilated by
$\mf{g}$ (from the structure theory of finite-dimensional
$\mf{sl}_2$-modules; see \cite{Hu}). Hence, $a$ is a central element.
\end{proof}\hfill
 
\begin{remark}
Using the anti-involution $j$, we can get a similar description of
elements which commute with $f$, as an algebra generated by $f,\Delta$.
\end{remark}

Recall that $\Delta=h^2+4ef-2h$ is the Casimir element. Our next step
will be to compute the commutators of powers of $\Delta$, with $x$ and
$y$. For $n=1$, we have
\begin{align*}
[\Delta,x] & = hx + xh + 4ey - 2x = (2h-3)x + 4ey,\\
[\Delta,y] & = -hy -yh + 4xf + 2y = -(2hy+y) + 4(fx-y) + 2y\\
& = (-2h-3)y + 4fx.
\end{align*}

\noindent (Note that $(\ad x)^3(\Delta) = (\ad y)^3(\Delta) = 0$ in
$H_0$.) We next extract information about these commutators.

\begin{prop}\label{P1}
There exist polynomials $f_n, g_n \in \Z[T] \subset k[T]$ for all $n$,
such that
\[ [\Delta^n, x] = (f_n(\Delta)h + g_n(\Delta))x + 2 f_n(\Delta) ey \]

\noindent and for $y$, we have
\[ [\Delta^n, y] = 2f_n(\Delta)f x + (g_n(\Delta) - f_n(\Delta)h) y. \]

\noindent The polynomials $f_n, g_n$ are inductively defined as follows:
\begin{eqnarray}\label{Eq1}
f_1(T) = 2, \qquad  && f_{n+1}(T) = 2 T^n + (T-1)f_n(T) - 2g_n(T),\\
g_1(T) = -3, \qquad && g_{n+1}(T) =  -3 T^n + (T +3) g_n(T) - 2T f_n(T).
\end{eqnarray}
\end{prop}

\begin{proof}
We show the various assertions made above.
\begin{enumerate}
\item Note for any $g \in \mf{Ug}$ that $[g, x]$ is, first, a
$\mf{Ug}$-linear combination of $x$ and $y$ only. Next, $[\Delta^n, x]$
is also a maximal vector for $\mf{g}$, of weight $1$. Thus, if it equals
$\alpha x + \beta y$, then $\alpha$ has weight $0$ and $\beta$ has weight
$2$. We now write $[e,[\Delta^n, x]] = 0$ to get
\[ 0 = [e, \alpha x + \beta y] = ([e,\alpha] + \beta)x + [e,\beta]y.\]

\noindent By the PBW theorem, the coefficients of $x,y$ therefore vanish.
Thus, $\beta \in \mf{Ug}$ is maximal of weight $2$, hence is a central
element times $e$ (by Lemma \ref{L2}).

Suppose we write $\beta = 2 f_n(\Delta) e$ for some polynomial $f_n$ in
$\Delta$. Then we get $[e,\alpha] + 2 f_n(\Delta) e = 0$, whence we get
that $\ad e(\alpha) = -2 f_n(\Delta) e$.

Since $[e,f_n(\Delta) h] = f_n(\Delta) \cdot (-2e)$, hence we see that
$\alpha - f_n(\Delta)h$ is killed by $e$. Moreover, it is a weight vector
of weight $0$, so it equals $g_n(\Delta)$ for some polynomial $g_n$.

Finally, the given initial values of $f_1, g_1$ do indeed satisfy the
commutation relations that we verified above.

\begin{remark}
We will sometimes omit $\Delta$ from $f_n(\Delta)$, but this should not
cause any confusion.
\end{remark}

\item We now compute the polynomials $f_n, g_n$ inductively. We have
\begin{eqnarray*}
[\Delta^{n+1},x] & = & \Delta^n[\Delta,x] + [\Delta^n,x]\Delta\\
& = & \Delta^n(2h-3)x + \Delta^n 4ey\\
&& + (f_n(\Delta)h + g_n(\Delta)) x \Delta + 2f_n(\Delta) ey \Delta\\
& = & (\Delta^n(2h-3) + f_n(\Delta)\Delta h + g_n(\Delta)\Delta)x\\
&& + (2f_n(\Delta)\Delta e + \Delta^n4e)y\\
&& - (f_n(\Delta)h + g_n(\Delta))(2h-3)x - 2f_n(\Delta)e4fx\\
&& - (f_n(\Delta)h + g_n(\Delta))4ey - 2f_n(\Delta)e(-2h-3)y.
\end{eqnarray*}

\noindent Grouping all elements containing $y$, we get the coefficient of
$y$ to be
\begin{align*} 
& 4\Delta^n e + 2f_n(\Delta)\Delta e - 4(f_n(\Delta) h + g_n(\Delta))e +
4 f_n(\Delta) eh + 6 f_n(\Delta)e\\
= & 2(2\Delta^n - 2g_n(\Delta)) e + 2 f_n(\Delta)(\Delta e - 2he + 2eh
+ 3e)\\
= & 2(2\Delta^n + f_n(\Delta)(\Delta - 1) - 2g_n(\Delta)) e,
\end{align*} 

\noindent whence we get that the coefficient of $y$ is
\[ 2f_{n+1}(\Delta)e = 2(2\Delta^n + f_n(\Delta)(\Delta - 1) -
2g_n(\Delta)) e. \]

\noindent This proves the relation for $f_{n+1}$. Similarly, grouping all
elements containing $x$, we get the coefficient of $x$ to be
\begin{eqnarray*}
&& \Delta^n(2h-3) + f_n(\Delta)\Delta h + g_n(\Delta)\Delta -
g_n(\Delta)2h\\
& + & 3g_n(\Delta) - 2f_n(\Delta) h^2 + 3 f_n(\Delta) h - 8 f_n(\Delta)
ef.
\end{eqnarray*}

\noindent Note that the sum of the last three terms is $- f_n(\Delta) h -
2f_n(\Delta)\Delta$. Hence we get that the coefficient is
\begin{eqnarray*}
&& f_{n+1}(\Delta) h + g_{n+1}(\Delta)\\
& = & \Delta^n (2h-3) + f_n(\Delta)(\Delta h - h - 2 \Delta) +
g_n(\Delta)(\Delta - 2h +3).
\end{eqnarray*}

Subtracting $f_{n+1}(\Delta) h$ from both sides (and using the formula
above), we conclude that
\begin{align*} 
g_{n+1}(\Delta) = & \Delta^n (2h-3) + f_n(\Delta)(\Delta h - h - 2
\Delta) + g_n(\Delta)(\Delta - 2h +3)\\
& - \Delta^n 2h - f_n(\Delta)(\Delta h - h) +
2 g_n(\Delta) h\\
= & -3\Delta^n - 2f_n(\Delta)\Delta + g_n(\Delta)(\Delta + 3).
\end{align*}

\noindent Thus, we have shown the inductive formulae.\medskip

\item Computations with $y$ are directly analogous to the ones above.
%
%
%
%
%
%
\end{enumerate}
\end{proof}

As a corollary of these calculations, we have

\begin{cor}\label{C1}\hfill
\begin{enumerate}
\item $f_n$ and $g_n$ are polynomials of degree $n-1$, with top
coefficients $2n$ and $-n(2n+1)$ respectively.

\item The $f_n$'s (or $g_n$'s) form a basis of $\mf{Z}(\mf{Ug})$.

\item The only elements from $\mf{Z}(\mf{Ug})$ that commute with $x$ or
$y$ are scalars.
\end{enumerate}
\end{cor}

\begin{proof}\hfill
\begin{enumerate}
\item (At first, recall that $\mf{Z}(\mf{Ug})$ is generated by $\Delta$;
see \cite{Hu}.)
All these facts are proved simultaneously by induction on $n$; they
clearly hold for $n=1$. Suppose they now hold for $n$. The inductive
definitions then show that $f_{n+1}$ has leading term arising from $2T^n
+ T \cdot (2n T^{n-1} + \cdots)$. Hence $f_{n+1} = 2(n+1) T^n + \cdots$.

Similarly, the top coefficient of $g_{n+1}$ is the coefficient of $T^n$
(unless it vanishes), and this equals
\[ -3 -n(2n+1) - 2 \cdot 2n = -(3 + 2n^2 + n + 4n) = -(n+1)(2(n+1)+1) \]
as claimed. Hence we are done by induction.

\item This is because both denote a unipotent change of basis from the
usual $\{ 1, T, T^2, \dots \}$, and the map sending $T$ to $\Delta$ is an
isomorphism : $k[T] \to \mf{Z}(\mf{Ug})$.

\item Note that an element from $\mf{Z}(\mf{Ug})$ commutes with $x$ if
and only if it commutes with $y$ (applying the anti-involution $j$ and
noting that $j$ fixes $\Delta$). Thus, we need to show that if
$\sum_{i>0} a_i\Delta^i$ commutes with $x$, it must be 0. But we have
\[ \sum_i a_i [\Delta^i, x] = \left( \sum_i a_i (f_i(\Delta) h +
g_i(\Delta)) \right) \cdot x + 2 \left( \sum_i a_i f_i(\Delta) \right) ey.
\]

\noindent Both coefficients (i.e., of $x$ and $y$) must therefore be
zero. Since the associated graded of $H_z$ is an integral domain, hence
$\sum_i a_i f_i(\Delta) = 0$; since the $f_i$'s form a basis of the
center, we get each $a_i$ to be zero, and we are done.
\end{enumerate}
\end{proof}

We have the following proposition, which will be used later.

\begin{prop}\label{P2}
Suppose $\psi, \eta, \alpha, \beta$ are central in $\mf{Ug}$. Then the
following are equivalent:
\begin{enumerate}
\item $2 \psi ey + (h \psi + \eta)x = [\alpha, x] + \beta x$.
\item $2 \psi fx + (\eta - h \psi)y = [\alpha, y] + \beta y$.
\item $\psi = \sum_{i>0} a_i f_i(\Delta)$, $\alpha = \sum_{i \geqslant 0}
a_i \Delta^i$,  $\beta = \eta - \sum_i a_i g_i(\Delta)$ for some scalars
$a_i\in k$.
\end{enumerate}
\end{prop}

\noindent Thus, either of the first two equations has a unique solution
in $\alpha, \beta$ (modulo the constant term in $\alpha$).

\begin{proof}
We first prove that the last statement implies the first two. Given
$\psi, \eta$ and $\alpha, \beta$ as in the last part, we compute:
\begin{align*}
[\alpha, x] + \beta x = & \sum_i [a_i (f_i(\Delta)h + g_i(\Delta))x + 2
a_i f_i(\Delta) ey] + \eta x - \sum_i a_i g_i(\Delta) x\\
= & \sum_i a_i f_i(\Delta) \cdot hx + 2 \sum_i a_i f_i(\Delta) \cdot ey
+ \eta x\\
= & 2 \psi ey + (h \psi + \eta)x.
\end{align*}

Similarly, 
\begin{align*}
[\alpha, y] + \beta y = & \sum_i [a_i (g_i(\Delta) - f_i(\Delta)h)y + 2
a_i f_i(\Delta) fx] + \eta y - \sum_i a_i g_i(\Delta) x\\
= & -\sum_i a_i f_i(\Delta) \cdot hy + 2 \sum_i a_i f_i(\Delta) \cdot
fx + \eta y\\
= & 2 \psi fx + (\eta - h \psi)y.
\end{align*}

To prove that the first two parts imply the last, we first note that the
solution set $\alpha, \beta$ is ``additive" in the variables $\psi,
\eta$. Therefore it suffices to show that if $[\alpha, x] + \beta x = 0$
or $[\alpha, y] + \beta y = 0$, then $\alpha = \beta = 0$.

So suppose $\alpha = \sum_{i \geqslant 0} a_i f_i(\Delta)$. Computing the
above expressions, we have
\[ [\alpha, x] + \beta x = \sum_{i>0} [a_i (f_i(\Delta)h + g_i(\Delta))x
+ 2 a_i f_i(\Delta) ey] + \beta x.\]

\noindent Equating the coefficient of $y$ to zero, since the
$f_i(\Delta)$'s form a basis of the center, and since $H_z$ is an
integral domain, we get that $a_i = 0\ \forall i$, so $\alpha = a_0 \in
k$. But then we are left with $\beta x = 0$, whence $\beta = 0$ too.

A similar proof is for the other equation, using the computations:
\[ [\alpha, y] + \beta y = \sum_{i>0} [a_i (g_i(\Delta) - f_i(\Delta)h)y
+ 2 a_i f_i(\Delta) fx] + \eta y - \sum_{i>0} a_i g_i(\Delta) \cdot x. \]
\end{proof}

\begin{prop}\label{Pfg}
The polynomials $f_n, g_n$ satisfy the recursive relations
\begin{eqnarray*}
f_1(T) & = & 2, \qquad f_2(T) = 4(T+1),\\
f_{n+2}(T) & = & (2T+2)f_{n+1}(T) - (T^2 - 2T-3) f_n(T),\\
g_1(T) & = & -3, \qquad g_2(T) = -10T-9,\\
g_{n+2}(T) & = & (2T+2) g_{n+1}(T) - (4T^{n+1} + 3T^n) - (T^2 - 2T-3)
g_n(T).
\end{eqnarray*}
\end{prop}

\begin{proof}
The initial values of $f_1, f_2, g_1, g_2$ can be computed easily using
Proposition \ref{P1} above. We now compute the expressions for $f_n,
g_n$.

Multiplying the equation in Proposition \ref{P1} for $f_n$ by $(T+3)$,
and that for $g_n$ by 2, and adding these up, the coefficients of $g_n$
on the right cancel each other. Hence we get
\begin{eqnarray*}
&& (T+3) f_{n+1}(T) + 2g_{n+1}(T)\\
& = & 2T^{n+1} + 6 T^n + (T-1)(T+3)f_n(T) - 6T^n - 4T f_n(T)\\
& = & 2T^{n+1} + f_n(T)(T^2 - 2T-3).
\end{eqnarray*}

But equation \eqref{Eq1} for $f_n$ also gives us an expression for $2
g_n(T)$ in terms of the $f_n$'s. Hence
\[ 2 g_{n+1}(T) = 2 T^{n+1} - f_{n+2}(T) + (T-1) f_{n+1}(T). \]

Replacing this in the previous equation, we get
\begin{eqnarray*}
&& (T+3) f_{n+1}(T) + 2 T^{n+1} - f_{n+2}(T) + (T-1) f_{n+1}(T)\\
& = & 2T^{n+1} + f_n(T)(T^2 - 2T-3),
\end{eqnarray*}

\noindent from which the relevant equation follows.

We now show the analogous result for $g_n$. Multiply the equation in
Proposition \ref{P1} for $f_n$ by $2T$, and that for $g_n$ by $(T-1)$. If
we now add the two, the coefficients for $f_n$ cancel each other, and we
get
\begin{eqnarray*}
&& 2T f_{n+1}(T) + (T-1) g_{n+1}(T)\\
& = & 4 T^{n+1} - 4T g_n(T) - 3T^{n+1} + 3T^n + (T+3)(T-1) g_n(T)\\
& = & (T^{n+1} + 3 T^n) + g_n(T)(T^2 - 2T - 3).
\end{eqnarray*}

Once again, equation \eqref{Eq1} for $g_n$ also gives us an expression
for $2T f_{n+1}(T)$ (after a change of variables), namely,
\[ 2T f_{n+1}(T) = -3T^{n+1} + (T+3) g_{n+1}(T) - g_{n+2}(T). \]

\noindent Substituting in the previous equation, and rearranging terms,
we obtain
\begin{eqnarray*}
g_{n+2}(T) & = & -3 T^{n+1} - (T^{n+1} + 3 T^n) +
g_{n+1}(T)((T-1)+(T+3))\\
&& - g_n(T)(T^2-2T-3)\\
& = & -(4T^{n+1} + 3 T^n) + (2T+2)g_{n+1}(T) - g_n(T)(T^2-2T-3).
\end{eqnarray*}
\end{proof}

We end this subsection by explicitly computing $f_n$ and $g_n$, though we
will not use this anywhere else in the paper.

\begin{lemma}
For all $n \geq 0$, we have
\begin{eqnarray*}
f_n(T) & = & \half (T+1)^{\frac{n-1}{2}} [x_+^n - x_-^n],\\
g_n(T) & = & T^n - \half (T+1)^{\frac{n-1}{2}} \left[ (\sqrt{T+1} + 1)
y_+^n + (\sqrt{T+1} - 1) y_-^n \right],
\end{eqnarray*}
where $x_\pm := \sqrt{T+1} \pm 1$, and $y_\pm := \sqrt{T+1} \pm 2$.
\end{lemma}

\begin{proof}
The claim is verified for the $f_n$'s by induction (using: $P(n-2),
P(n-1) \Rightarrow P(n)$). Similarly, to verify the claim for the
$g_n$'s, we first define $h_n(T) = (g_n(T) - T^n) / T^{n-2}$ $\in
\mathbb{Z}[T,T^{-1}]$; one now shows that the equation for the $g_n$'s in
Proposition \ref{Pfg} is equivalent to: $h_1 = -T(T+3),\ h_2 =
-(T+1)(T+9)$, and
\[ h_{n+2} = 2 \left( \frac{T+1}{T} \right) h_{n+1} -
\frac{(T+1)(T-3)}{T^2} h_n. \]

\noindent One checks by induction, that the given formula solves this
system.
\end{proof}\hfill

\subsection{A central element that generates the center}
   
Recall that we wanted to write $\omega = [x, t - \half hz]$ as a
commutator of $x$ with a central element of $\mf{Ug}$ (see the remarks
after equation \eqref{Ereqd}). We first claim that $\omega$ can be
rewritten as $z[\half\Delta,x]-(e[z,y]+\half h[z,x])+\half[z,x]$. Indeed,
we can simplify this expression to get
\begin{eqnarray*}
&& \half z ((2h-3)x + 4ey) -(e[z,y]+\half h[z,x])+\half[z,x]\\
& = & (2ezy - \frac{3}{2} zx + hzx) + \half[z,x] - e[z,y] - \half h[z,x],
\end{eqnarray*}

\noindent which equals the expression used to define $\omega$. We
therefore work with this new expression, and further rewrite it as
\begin{eqnarray*}
\omega & = & z[\half \Delta, x] - ([z, ey] + [z, \half hx]) +
\half[z,x]\\
& = & z[\half \Delta, x] - \half [z, 2ey + hx - x]\\
& = & z[\half \Delta, x] - \frac{1}{4}[z, 4ey + (2h-3)x + x]\\
& = & z[\half \Delta, x] - \frac{1}{4}[z, [\Delta, x] + x]\\
& = & \frac{1}{4}(2z [\Delta, x] - z[\Delta, x] + [\Delta, x] z -
[z,x])\\
& = & \frac{1}{4}(z[\Delta,x] + [\Delta,x]z - [z,x])\\
& = & \frac{1}{4}[x, z - \Delta z] + \frac{1}{4}(z[\Delta, x] -
\Delta[z,x]).
 \end{eqnarray*}

We would like to show that $\omega = [x,q_z]$ for some $q_z \in
\mf{Z}(\mf{Ug})$. If we can now show that there exists $z_0 \in
\mf{Z}(\mf{Ug})$ such that $[z_0, x] = z[\Delta,x] - \Delta[z,x] = \Delta
xz - zx\Delta$, then $t - \half hz - q_z$ would be central, where
\begin{equation}\label{Eqz}
q_z = \frac{1}{4} z - \frac{1}{4} \Delta z - \frac{1}{4} z_0.
\end{equation}

\noindent The existence of $z_0$ follows from the following result,
setting $z' = \Delta$:

\begin{prop}\label{Pzz}
Given $z, z' \in \mf{Z}(\mf{Ug})$, we can find $z_0 = z_0(z,z') \in
\mf{Z}(\mf{Ug})$ such that
\[ [z_0, x] = zxz' - z'xz = z'[z,x] - z[z',x]. \]

\noindent and $\displaystyle z_0(c_1 \Delta^m + l.o.t.,\ c_2 \Delta^n +
l.o.t.) = c_1 c_2 \left( \frac{m-n}{m+n} \right) \Delta^{m+n} + l.o.t.$
\end{prop}

(Here, lower order terms are smaller powers of $\Delta$.)

\begin{proof}
First, it is easy to see that the ``solution" $z_0$ is ``bilinear" in $z,
z'$, in that $z_0(z+r,z'+s) = z_0(z,z') + z_0(z,s) + z_0(r,z') +
z_0(r,s)$ for all $z,r,z',s$ central in $\mf{Ug}$. It therefore suffices
to prove the result for $z = \Delta^m, z' = \Delta^n$ for some $m,n \geq
0$. But then we have
\begin{eqnarray*}
zxz' - z'xz & = & \Delta^n[\Delta^m, x] - \Delta^m [\Delta^n, x]\\
& = & \Delta^n [(f_m(\Delta)h + g_m(\Delta))x + 2 f_m(\Delta) ey]\\
&& - \Delta^m [(f_n(\Delta)h + g_n(\Delta))x + 2 f_n(\Delta) ey]\\
& = & 2 \psi ey + (h \psi + \eta)x,
\end{eqnarray*}

\noindent where
\[ \psi = \Delta^n f_m(\Delta) - \Delta^m f_n(\Delta), \qquad \eta =
\Delta^n g_m(\Delta) - \Delta^m g_n(\Delta). \]

In particular, the top degree and coefficient of $\psi$ can be computed
from Corollary \ref{C1}:
\begin{equation}\label{Epsi}
\psi = (2m-2n) \Delta^{m+n-1} + l.o.t..
\end{equation}

\noindent By Proposition \ref{P2} above, $zxz' - z'xz = [\alpha, x] +
\beta x$ for some central $\alpha, \beta \in \mf{Ug}$.

Let us also evaluate $\Delta^m y \Delta^n - \Delta^n y \Delta^m$. We get
\begin{eqnarray*}
zyz' - z'yz & = & \Delta^n[\Delta^m, y] - \Delta^m [\Delta^n, y]\\
& = & \Delta^n [(g_m(\Delta) - f_m(\Delta)h)y + 2 f_m(\Delta) fx]\\
&& - \Delta^m [(g_m(\Delta) - f_n(\Delta)h)y + 2 f_n(\Delta) fx]\\
& = & 2 \psi' fx + (\eta' - h \psi')y,
\end{eqnarray*}

\noindent where
\[ \psi' = \Delta^n f_m(\Delta) - \Delta^m f_n(\Delta) = \psi, \qquad
\eta' = \Delta^n g_m(\Delta) - \Delta^m g_n(\Delta) = \eta. \]

\noindent By Proposition \ref{P2}, this equals $[\alpha, y] + \beta y$
for the same $\alpha, \beta$ as above. Thus, 
\[ zxz' - z'xz = [\alpha, x] + \beta x, \qquad zyz' - z'yz = [\alpha, y]
+ \beta y, \]

\noindent where $z,z' \in \mf{Z}(\mf{Ug})$. We now prove that $\beta =
0$, as desired. Applying the anti-involution $j$ to the second of the
equations, and noting that $j$ preserves $z,z'$ (since it preserves
$\Delta$) and sends $y$ to $x$, we get
\[ z'xz - zxz' = [x, \alpha] + x \beta. \]

\noindent Comparing with the (negative of the) first equation, we see
that
\[ [x, \alpha] - \beta x = [x, \alpha] + x \beta, \]

\noindent whence we conclude that
\[ 0 = \beta x + x \beta = 2 \beta x - [\beta, x], \]

\noindent and the uniqueness result in Proposition \ref{P2} implies
$\beta = 0$, as claimed. Hence $zxz' - z'xz = [\alpha, x]$, as desired.

Moreover, to show the last equation, it suffices by (bi)linearity of
$z_0$ to show that $z_0(\Delta^m, \Delta^n)$ is of the desired form (with
$c_1 = c_2 = 1$). But $z_0(\Delta^m, \Delta^n) = \alpha$ above, so we
need to compute the top degree and coefficient of $\alpha$. This comes
from $\psi$ (equation \eqref{Epsi}) and the ``unipotent" change of basis
from the $f_n$'s to the $\Delta^{n-1}$'s (Corollary \ref{C1}). Thus,
$\psi = \frac{2m-2n}{2(m+n)} f_{m+n} + l.o.t.$. Now use Proposition
\ref{P2} to get that
\[ z_0(\Delta^m, \Delta^n) = \alpha = \frac{2m-2n}{2(m+n)} \Delta^{m+n} +
l.o.t.. \]
\end{proof}

As a consequence, we have information about $q_z$ (see equation
\eqref{Eqz} above):

\begin{cor}\label{C3}
For any $z = c \Delta^m + l.o.t.$, $q_z = \frac{-cm}{2(m+1)} \Delta^{m+1}
+ l.o.t.$.
\end{cor}

\begin{proof}
From equation \eqref{Eqz}, the top term of $q_z$ comes from the last two
terms, since $z_0 = z_0(z,\Delta)$ here. If $z = c \Delta^m + l.o.t.$
here, then by Proposition \ref{Pzz}, the top term is
\[ -\frac{1}{4}c \Delta^{m+1} - \frac{1}{4} c \left( \frac{m-1}{m+1}
\right) \Delta^{m+1} \]

\noindent and this simplifies to the desired form.
\end{proof}

This shows us that the center of $H_z$ is nonempty and contains an
element of the form
\[ t_z = t - \half hz - \frac{1}{4} z + \frac{1}{4} \Delta z +
\frac{1}{4} z_0, \]
where $z_0 = z_0(\Delta, [x,y])$ as in the above results.

\subsection{Various centralizers and the center}

It just remains to prove that this element $t_z$ generates the whole
center of $H_z$. We do this in steps.
First, we describe the elements of $H_z$ which commute with various sets.

\begin{prop}\label{P3}\hfill
\begin{enumerate}
\item The centralizer in $H_z$ of $\mf{Ug}$ is freely generated by
$\Delta, t_z$.

\item
\begin{enumerate}
\item The centralizer of $e$ (i.e., the set of $\mf{sl}_2$-maximal
vectors) in $H_z$ is the subalgebra generated by $\Delta, t_z, e, x$.

\item The centralizer of $e$ and $x$ (together) in $H_z$ is freely
generated by $t_z, e, x$.
\end{enumerate}

\item The centralizer of $V$ in $H$ (for $z=0$) is freely generated by
$t, x, y$.
\end{enumerate}
\end{prop}

\noindent Using the anti-involution $j$, we get similar results involving
$f,y$.

\begin{proof}
In all but the last part, it is enough to show that the prescribed
elements generate the centralizer (call it $B$ for this paragraph and the
next) in $H$ (i.e., when $z=0$). This is because all ``claimed
generators" ($\Delta, t_z, e, x$) in $H$ have lifts to $H_z$, and any $b
\in B$ has a principal symbol in $H$, a lift of which can be subtracted
from $b$ to get $b' \in B$ of ``smaller filtration degree" in $V$. Now
proceed by induction.

Moreover, that the prescribed elements freely generate $B \subset H_z$
(except possibly for $H_z^e$) would follow from the corresponding
statement for $z=0$, since any relation among the lifts in $H_z$ gives a
relation in $H$. Let us start by showing that various elements are
algebraically independent in $H$.

We first note that $t = t_0, \Delta$ are algebraically independent in
$H$, for if $\sum a_{ij} t^i \Delta^j$ $= 0$, then checking the
coefficients of $x,y$ (via \eqref{Eeqn} below) gives the result.

Next, we claim that $t,e,x$ are algebraically independent in $H$.
Indeed, if $\sum a_{q,r,s} e^q t^r x^s = 0$, then consider the highest
power of $y$ that occurs (i.e., $2r$ for the highest $r$); then for this
$r$, consider the highest power of $e$. Now for these, the highest power
of $x$ must have coefficient $a_{qrs} = 0$.

Finally, $t,x,y$ are algebraically independent, for if $\sum a_{qrs} t^q
x^r y^s = 0$, then writing this element in terms of the ordered PBW-basis
$(e,f,h,x,y)$, we can conclude that $a_{qrs} \equiv 0$.

\begin{enumerate}
\item By passing to the associated graded, it is enough to show the
proposition for $z=0$; thus, we assume that $H_z=H$.

Let $a$ be an element in $H$ which is in the centralizer of $\mf{Ug}$;
without loss of generality, we may assume $a$ to be a weight vector for
$\ad h$ and to be homogeneous in $x$ and $y$, by decomposing it into such
components (since $\ad \mf{g}$ preserves this grading degree).

Writing $a$ as a polynomial in the PBW basis above, let $n$ be the
smallest power of $x$ appearing in this polynomial. Thus, $a = b x^n$ for
some $b \in H$.

Since $H$ is an integral domain, and $a,x$ commute with $e$, so does $b$,
by the ``dividing trick" \eqref{Etrick}.
Since $[h,a] = 0$, hence $a$ is in the 0 weight space, whence the weight
of $b$ is $-n$. But no maximal vectors in $H$ may have a negative weight
(by $\mf{sl}_2$-theory and Lemma \ref{L0}), whence $n=0$.

Now let us look at the monomial term of $a$ with the highest power of
$y$. Since $a$ is not divisible by $x$ and is homogeneous, this term must
be of the form $cy^m$, for some $c \in \mf{Ug}$.
We claim that $[e,c] = 0$. This is because $[e,a] = 0$, and upon applying
$\ad e$, the power of $x$ in a monomial cannot decrease, and the power of
$y$ cannot increase. So we have
\[ 0 = [e,a] = [e,c]y^m + c [e,y^m] + [e, \dots], \]

\noindent and the only monomial with no $x$'s and $m$ $y$'s in it, is
$[e,c] y^m$.

We thus get that $c$ is maximal in $\mf{Ug}$, of weight $m$. Thus $m$ is
even, and $c$ is of the form $e^{m/2} \alpha$ for some central $\alpha$,
by Lemma \ref{L2}.

Let us now consider $a - \alpha t^{m/2}$. By \eqref{Eeqn} below, the
monomials in either term of highest $y$-degree, are $\alpha e^{m/2} y^m$.
Therefore $a - \alpha t^{m/2}$ has highest power of $y$ (without any
power of $x$) in a monomial, strictly less than $m$. Arguing inductively,
we get down to when $m=0$, leaving us with a vector in $\mf{Ug}$. This
commutes with $\mf{Ug}$, so it is central in $\mf{Ug}$, and we are
done.\medskip

\item Once again, we may assume that $z=0$. Let $a$ be a {\it weight}
vector that commutes with $e$; we may assume that it is also homogeneous
(in $V$, say of degree $k$, on which we will do induction) and not
divisible by $x$ from the right (by the ``dividing trick"
\eqref{Etrick}). Hence it may be written as $a = \sum_{0 \leq i \leq k}
c_i y^i x^{k-i}$, where $c_i \in \mf{Ug}\ \forall i$ (and $c_k \neq 0$).
Moreover, $[e,a] = 0$ yields: $c_i = -[e,c_{i-1}/i]$, whence $c_i = (\ad
(-e))^i(c_0) / i!$. In particular, $c_0 \neq 0$ as well.

Now consider $c_k$; we claim that $c_k$ is $\mf{g}$-maximal too, since
$[e,c_k]$ is the coefficient of $y^k$ in $[e,a]$. By Lemma \ref{L2}, $c_k
= \alpha e^n$ (since $a$ is a weight vector), with $\alpha \in
\mf{Z}(\mf{Ug})$.

\begin{enumerate}
\item We now prove this part by induction on $k$. The base case of $k=0$
follows from Lemma \ref{L2}. Now continue with the above analysis. Note
that $\alpha e^n \in \mf{Ug}$ is $\mf{sl}_2$-maximal of weight $2n$,
whence $(-1)^k / k! \cdot (\ad e)^k(c_0) = \alpha e^n \notin (\ad
e)^{2n+1}(\mf{Ug})$. In particular, $k \leq 2n$, whence $n \geq \lceil
k/2 \rceil$. We now have two cases:
\begin{itemize}
\item If $k$ is even, we define $b := a - \alpha e^{n-(k/2)} t^{k/2}$.

\item If $k$ is odd, we use $[\Delta,x] =$ (up to scaling) $[4fe + h^2 +
2h, x] = 4 ey + 2hx - x$. In this case, we define
\[ b := a - \frac{1}{4} \alpha e^{n - \lceil k/2 \rceil} t^{\lfloor k/2
\rfloor} \cdot [\Delta, x]. \]
\end{itemize}

\noindent In both cases, $b \in H \cdot x$ by \eqref{Eeqn} below, and by
the ``dividing trick" \eqref{Etrick}, the quotient is a weight vector
with {\it smaller} degree of homogeneity (in $V$), so we are done by
induction.\medskip

\item We continue from where we had stopped before the previous sub-part.
Now suppose that $a$ commutes with $x$ as well. Then $\alpha_y = 0$,
where $[\alpha, x] = \alpha_x x + \alpha_y y$ (looking at the coefficient
of $y^{k+1}$). But by Proposition \ref{P1}, this can only happen if
$\alpha$ is a constant; let us suppose it is 1. Thus, we have $c_k = e^n
= (-1)^k /k! \cdot (\ad e)^k c_0$, whence $c_0$ has weight $2(n-k)$.

Next, note that if $[c_0, x] = rx + sy$ with $r,s \in \mf{Ug}$, then
$r=0$ by considering the coefficient of $x^{k+1}$ in $[a,x] = 0$. Now
suppose that we write $c_0 = \sum_i e^{n-k+i} f^i p_i(h)$ for polynomials
$p_i$. We claim that the $p_i$'s are constant, for otherwise
\begin{eqnarray*}
[c_0, x] & = & \sum_i e^{n-k+i} \left( i f^{i-1} y \cdot p_i(h) + f^i
[p_i(h), x] \right)\\
& = & \sum_i e^{n-k+i} \left( i f^{i-1} p_i(h+1) y + f^i (p_i(h) -
p_i(h-1))x \right),
\end{eqnarray*}

\noindent and this is not in $\mf{Ug} \cdot y$ as claimed above.

Thus, we have $c_0 = \sum_{i=0}^N \alpha_i e^{n-k+i} f^i$, say. Now
consider a general situation in $\mf{U}(\mf{sl}_2)$: repeatedly applying
$\ad e$ to $f^i$ for any $i$ can only lead to $e^j$ (up to a scalar) if
$j=i$; and then $\ad e (e^i) = 0$. On the other hand, if $(\ad
e)^{2i}(f^i) \in k^\times \cdot e^i$, then $(\ad e)^j (f^i)$ is not a
power of $e$ if $j<2i$ (else $(\ad e)^{j+1}(f^i) = 0$), and vanishes if
$j>2i$.

Thus, if we now consider the ``last" summand in $c_0$, $(\ad e)^k$ must
send $f^N$ to $e^{k-N}$, in order that we get $e^{n-k + N + k-N} = e^n$.
But then $k = 2N$, and we get that
\[ c_0 = \alpha_{k/2} e^{n-k/2} f^{k/2} + \dots + \alpha_0 e^{n-k} f^0 \]

\noindent and $c_i = \ad(-e)^i(c_0) / i!$ is also divisible by $e^{n-k}$
for all $i$. Hence taking $e^{n-k}$ common on the left, we get that
\[ a = e^{n-k} \left( e^k y^k + \dots + (\alpha_{k/2} e^{k/2} f^{k/2} +
\dots + \alpha_0) x^k \right). \]

In particular, by the ``dividing trick" \eqref{Etrick}, the terms in the
parentheses commute with $e,x$. We can divide by $e^{n-k}$ and then
subtract $e^{k/2} t^{k/2}$.

Now note (as an aside) that $t = ey^2 + (hy + fx)x$, so that $t^n -
(ey^2)^n \in H \cdot x\ \forall n$. It is also easy to check that
$(ey^2)^n - e^ny^{2n} \in H \cdot x$ (e.g., by induction on $n$). Thus,
\begin{equation}\label{Eeqn}
t^n - e^n y^{2n} \in H \cdot x\ \forall n.
\end{equation}

\noindent In particular, $e^{k/2} t^{k/2} - e^k y^k \in H \cdot x$. Using
the ``dividing trick" \eqref{Etrick}, dividing this by $x$ yields a
maximal vector $a'$ that commutes with $e,x$, is a weight vector, and is
homogeneous of {\it smaller} degree than $k$, whence we are done by
induction.

It remains to check the base case; but $k=0$ would mean the centralizer
of $e,x$ in $\mf{Ug}$, and by Lemma \ref{L2} and properties of
$[\Delta^n,x]$, the only such elements are polynomials in $e$.
\end{enumerate}\medskip

\item Since both sides of the desired equality are $(\ad)
\mf{g}$-submodules of $H$ (and $H$ is a direct sum of finite-dimensional
$\mf{g}$-modules by Corollary \ref{C0}), it would suffice to show that
any $\mf{g}$-maximal vector from $H^V$ belongs to $\mf{Z}(H) \Sym V$. By
the previous part, this consists of the $y$-centralizer of $H^{\{ e,x \}}
= k[t,e,x]$. Since $t,x$ are in this centralizer, say $\sum_i r_i(t,x)
e^i$ commutes with $y$. Thus,
\[ \sum_i r_i(t,x) i e^{i-1} x = \left[ \sum_i r_i(t,x) e^i, y \right] =
0 \]

\noindent and by the algebraic independence of $t,e,x$, we are done.
\end{enumerate}
\end{proof}

We can finally conclude the proof of the first part of Theorem \ref{T1}
above.

\begin{proof}[Proof of the first part]
Let $a$ be a central element of $H_z$. In particular, it commutes with
$\mf{g}$, so by Proposition \ref{P3}, it can be written as a polynomial
in $t_z$ with coefficients in $\mf{Z}(\mf{Ug})$. Let $\kappa t_z^n$ be a
monomial of top degree. Since $[x,a]=0$, passing to the associated graded
ring (with respect to the filtration), we get that $[x,\kappa]=0$.

By Corollary \ref{C1} above, $\kappa$ is a scalar; so we may disregard
the top term of $a$. Continuing by induction, we see that all
coefficients of $a$ are scalars. Hence the center of $H_z$ is generated
by $t_z$, and it is transcendental over $k$ if $t$ is transcendental in
$H = H_0$. But this follows by the PBW property.
\end{proof}\hfill

We conclude our discussion of the center by giving an explicit formula
for the central element when $z$ is (at most) linear. Suppose $[x,y] =
a\Delta + b$ for scalars $a,b$. We therefore want to produce $z_0$
central in $\mf{U}(\mf{sl}_2)$, such that
\[ [z_0, x] = \Delta x (a \Delta + b) - (a \Delta + b) x \Delta =
b(\Delta x - x \Delta) = b[\Delta,x]. \]

\noindent Therefore $z_0 = b \Delta$ works, and we have the central
element
\[ t'_z = ey^2 + hxy - f x^2 - \half h(a\Delta + b) + \frac{1}{4}(\Delta
(a \Delta + b) - (a \Delta + b) + b\Delta). \]

\noindent Removing the scalar $-5b/4$, we get the desired generating
central element to be (up to adding a scalar)
\[ t_z = e y^2 + hxy - fx^2 - \half h(a \Delta + b) + \frac{1}{4}(a
\Delta^2 + (2b-a)\Delta). \]

\section{Derivations and commutator quotient}

\subsection{Derivations}

We now compute the space of derivations. Note that if $D$ is a derivation
of $H$, then we may assume, modulo an inner derivation, that it vanishes
on $\mf{Ug}$, since $\mf{g}$ is simple. Thus, $D$ is a $\mf{g}$-module
map, so $D(x)$ is a maximal vector of weight 1. By Proposition \ref{P3},
it is of the form
\[ D(x) = \sum_{i \geq 0} b_i(t_z) r_i + c_i(t_z) s_i, \]

\noindent where $r_i := \Delta^i x,\ s_i := [\Delta^i,x]\ \forall i$.
But since we can rewrite the sum of half of these terms as
\[ \sum_i c_i(t_z) s_i = \sum_i c_i(t_z) [\Delta^i, x] = \left[ \sum_i
c_i(t_z) \Delta^i, x \right], \]

\noindent hence by subtracting another inner derivation, we may assume
that $D(x) = \sum_i b_i(t_z) \Delta^i \cdot x$. (Note that this change
does not affect the fact that $D \equiv 0$ on $\mf{Ug}$.) Let us also
denote $\sum_i b_i(t_z) \Delta^i$ by $\omega$.

We now compute $D(y)$: we claim that $D(y) = \omega y$. To see this,
apply $D$ to the relation $[e,y] = x$. Then
\[ [e, D(y)] = D(x) = \sum_i b_i(t_z) \Delta^i \cdot x, \]

\noindent whence it is easy to see that $[e, D(y) - \omega y] = 0$. Since
$D$ is now a $\mf{g}$-module map, hence $D(y)$, and thus $D(y) - \omega
y$, are both weight vectors of weight $-1$. But the last is also maximal,
from above. Hence it vanishes, i.e., $D(y) = \omega y$.

We also carry out a key computation, that we shall need later. Recall the
polynomials $f_n, g_n$ that came up while computing $[\Delta^n, x]$.

\begin{lemma}\label{L5}
For all $n$, we have
\[ [\Delta^n, x]y - [\Delta^n, y]x = 2f_n(\Delta) (ey^2 + hxy - fx^2 -
\half hz) + g_n(\Delta) z. \]
\end{lemma}

\begin{proof}
In what follows, we omit the $(\Delta)$, and refer to the polynomials
merely as $f_n, g_n$.
\begin{eqnarray*}
[\Delta^n, x]y - [\Delta^n, y]x & = & [2 f_n ey + (f_n h + g_n)x]y - [2
f_n fx + (g_n - f_n h)y]x\\
& = & f_n(2 ey^2 + h(xy + yx) - 2 fx^2) + g_n(xy - yx)\\
& = & f_n(2 ey^2 + h(2xy - z) - 2 fx^2) + g_n z,
\end{eqnarray*}

\noindent and hence we are done.
\end{proof}\hfill

We are now ready to finish the proof of the second part of Theorem
\ref{T1}.

\begin{prop}
If $z = 0$ (and $H = H_0$), then $Der(H) / Inn(H)$ is a rank one free
module over the center of $H$.
\end{prop}

\begin{proof}
Recall that $D(x)=\omega x,\ D(y)=\omega y$ and $\omega = \sum_i b_i(t)
\Delta^i$. We first claim that $b_i(t) = 0$ for $i>0$. Indeed, note that
\begin{eqnarray*}
D(xy) & = & \sum_i b_i(t)(\Delta^i xy + x \Delta^i y) = \sum_i b_i(t)( 2
\Delta^i xy - [\Delta^i, x] y),\\
D(yx) & = & \sum_i b_i(t)(\Delta^i yx + y \Delta^i x) = \sum_i b_i(t) (2
\Delta^i yx - [\Delta^i, y] x)
\end{eqnarray*}

\noindent and since $xy = yx$, hence one of the summands cancels
throughout, to give: $[\omega,x]y = [\omega,y]x$. Rewriting $\omega$ into
another different summation for convenience, we get an equation of the
form 
\[ \sum_{i=0}^m t^i [h_i(\Delta),x]y = \sum_{i=0}^m t^i [h_i(\Delta),
y]x. \]

Let $m$ be the highest index such that $h_m(\Delta)$ is not a constant.
We claim that this equation can not hold if $m>0$, since if we look at
the coefficient of $y^{2m+2}$, then the coefficient on the left side is
nonzero, whereas on the right side it is zero. This is a contradiction.

Thus we get $\omega = b(t) \in \mf{Z}(H)$, and $D(x) = \omega x$. We now
know the values of $D$ on generators, so using the Leibnitz rule, we can
now compute this map on all of $H$. Let us denote this map by $D_\omega$.
Since we have the PBW property (i.e., that $\mf{U}(\mf{g} \ltimes V)
\cong \mf{Ug} \otimes \Sym V$ as vector spaces), we observe that the map
$D_\omega$ is given by
\[ D_\omega(-) = n \omega \cdot - \mbox{, on } \mf{Ug} \otimes \Sym^n V\
\forall n \geq 0. \]

\noindent Moreover, it is not hard to verify that this defines a
derivation, using the PBW property again.

Finally, we verify that the map $: \mf{Z}(H) \to Der(H) / Inn(H)$,
sending $\omega \mapsto D_\omega$, is a vector space isomorphism, by
looking at $D_\omega(x)$, say (to verify linear independence). Hence
$H^1(H,H) = Der(H) / Inn(H) \cong \mf{Z}(H)$ as $\mf{Z}(H)$-modules, if
$[x,y] = 0$.
\end{proof}

Finally, we have the following proposition.

\begin{prop}
If $z \neq 0$, then every derivation of $H_z$ is inner.
\end{prop}

\begin{proof}
Note again that since we are working modulo $Inn(H_z)$, so that given a
derivation $D$, we assume $D$ kills $\mf{Ug}$ and $D(x) = \omega x,\ D(y)
= \omega y$ as above.

Let us write $D(x) = t_z^m h_m(\Delta) x + \sum_{0 \leq i < m} t_z^i
h_i(\Delta) x$. If we now pass to the associated graded algebra $\Gr H_z$
(under the usual filtration that assigns $V$ degree 1 and $\mf{g}$ degree
0), then we get a derivation of $\Gr H_z = \mf{U}(\mf{g} \ltimes V)$,
that sends $x$ to $t_z^m h_m(\Delta) x$. By the previous case, we may
assume without loss of generality that $h_m = 1$.

Similarly, $D(y) = t_z^m y + \sum_{i=0}^{m-1} t_z^i h_i(\Delta) y$.
Applying $D$  to $z$, we get
\begin{eqnarray*}
0 & = & D(z) = [Dx,y] + [x,Dy] = [\omega x, y] + [x, \omega y]\\
& = & 2 \omega [x,y] + [\omega, y]x - [\omega, x]y.
\end{eqnarray*}
We rearrange this to get
\[ 2 \omega z = [\omega, x]y - [\omega, y]x. \]

Let us rewrite $\omega = \sum_i b_i(t_z) \Delta^i$. Then using Lemma
\ref{L5}, we get
\begin{eqnarray*}
2 \omega z & = & \sum_i b_i(t_z) ([\Delta^i, y]x - [\Delta^i, x]y)\\
& = & \sum_i b_i(t_z) \left( 2 f_i(\Delta)(t - \half hz) + g_i(\Delta)z
\right).
\end{eqnarray*}

\noindent Also note, that $t_z = (t - \half hz) + \frac{1}{4}(\Delta z -
z + z_0)$, where $[z_0, x] = z[\Delta, x] - \Delta[z,x]$. Hence we
rewrite the above equation as
\begin{equation}\label{Eq3}
2 \omega z = 2\sum_i b_i(t_z) \left( f_i(\Delta)(t_z - \frac{1}{4}(\Delta
z - z + z_0)) + \half g_i(\Delta) z \right).
\end{equation}

Now look at the highest power of $t_z$ (or of $y$) in the equation, and
say the corresponding summand on the left side is $t_z^n \sum_j \beta_j
\Delta^j$, with $\beta_j \in k$. Then the corresponding expression on the
right side yields
\[ t_z^n \sum_j \beta_j \left( f_j(\Delta)(t_z - \frac{1}{4}(\Delta
z - z + z_0)) + \half g_j(\Delta) z \right). \]

Now note that there is an extra power of $t_z$ in this latter expression.
Therefore if we look at the highest power of $y$ that occurs in the right
side of equation \eqref{Eq3}, namely $y^{2n+2}$, then its coefficient
must be zero (since the corresponding coefficient on the left side is
zero). Since $t_z$ is central, this means that $\sum_j \beta_j
f_j(\Delta) = 0$. But the $f_j$'s form a basis of the center of
$\mf{Ug}$. Hence $\beta_j = 0$ for all $j$, whence $\omega$ must equal
zero too. We conclude that $D(x) = D(y) = D(\mf{g}) = 0$, and so $D = 0$
modulo $Inn(H_z)$, as claimed.
\end{proof}
This concludes the proof of Theorem \ref{T1}.\medskip

\subsection{Commutator quotient}

Next, we would like to determine the commutator quotient (or
abelianization) $H_z/[H_z,H_z]$ as a module over the center of $H_z$. At
first, let us consider the case $z=0$.

\begin{prop}\label{Pcenter}
The natural map from $\mf{Z}(\mf{Ug})$ to $H/[H,H]$ is an isomorphism,
and the action of the center of $H$ on its commutator quotient is
trivial.
\end{prop}

We need a small lemma for this, which is also used later.

\begin{lemma}\label{Lfilt}
Inside any $H_z$, we have $\mf{Ug} \cdot V = [\mf{Ug}, V]$. More
precisely, in terms of the standard filtration on $\mf{Ug}$,
$F^n \mf{Ug} \cdot V = [F^{n+1} \mf{Ug}, V]\ \forall n \geq 0$.
\end{lemma}

\begin{proof}
The second statement (for all $n$) implies the first; we will show both
inclusions for the latter claim. One way is easy: $[F^{n+1} \mf{Ug}, V]
\subset F^n \mf{Ug} \cdot V$ using induction on $n$.

For the other inclusion, we proceed by induction on $n$. Let $\alpha \in
F^n \mf{Ug}$; we want to show that $\alpha \otimes V \in [F^{n+1}
\mf{Ug}, V]$. When $n=0$, we are done since $[\mf{g},V]=V$, so it
suffices to show that
\[ \alpha \otimes x \in [F^{n+1} \mf{Ug}, V] \mod F^{n-1} \mf{Ug} \otimes
V. \]

\noindent But we have
\begin{eqnarray*}
{[} h^n,x] & \equiv & nh^{n-1}x \mod F^{n-1} \mf{Ug} \otimes V,\\
{[} h^n,y] & \equiv & -nh^{n-1}y \mod F^{n-1} \mf{Ug} \otimes V,\\
{[} f^n,x] & \equiv & nf^{n-1}y \mod F^{n-1} \mf{Ug} \otimes V,\\
{[} e^n,y] & \equiv & ne^{n-1}x \mod F^{n-1} \mf{Ug} \otimes V,
\end{eqnarray*}
so
\begin{eqnarray*}
{[}e^i h^j f^k, x] & \equiv & j e^i h^{j-1} f^k x + k e^i h^j f^{k-1} y
\mod F^{n-1} \mf{Ug} \otimes V,\\
{[}e^i h^j f^k, y] & \equiv & i e^{i-1} h^j f^k x - j e^i h^{j-1} f^k y
\mod F^{n-1} \mf{Ug} \otimes V.
\end{eqnarray*}

Now assume without loss of generality that $\alpha = e^i h^j f^k$, with
$i+j+k=n$. Then
\begin{eqnarray*}
\alpha \otimes x & \equiv & \frac{1}{j+1} [e^i h^{j+1} f^k, x] -
\frac{k}{j+1} e^i h^{j+1} f^{k-1} y \mod F^{n-1} \mf{Ug} \otimes V,\\
\alpha \otimes y & \equiv & \frac{-1}{j+1} [e^i h^{j+1} f^k, y] +
\frac{i}{j+1} e^{i-1} h^{j+1} f^k x \mod F^{n-1} \mf{Ug} \otimes V.
\end{eqnarray*}

\noindent We thus repeatedly (alternately) apply these two identities to
assume that either $i$ or $k$ becomes zero (in $\alpha$). Applying
(possibly both of) them once more, we are done.
\end{proof}

\begin{proof}[Proof of Proposition \ref{Pcenter}]
Since $H= \mf{U}(\mf{g} \ltimes V)$, from the relation between Lie
algebra homology and Hochschild homology, we get that
\begin{equation}\label{Ecoh}
H/[H,H] = H/[H,\mf{g} \ltimes V] = (H/[H,V])^{\mf{g}}.
\end{equation}

\noindent We now claim that $H/[H,V] = \mf{Ug}$, which would imply that $H
/ [H,H] = (\mf{Ug})^{\mf{g}} = \mf{Z}(\mf{Ug})$, as desired. Indeed,
obviously $\mf{Ug}$ injects into $H/[H,V]$, so we just need to
demonstrate that $HV \subset [H,V]$.

Clearly, $[H,V]$ is a right module over $\Sym V$, so it suffices to show
that $[\mf{Ug},V] \supset \mf{Ug} \otimes V$. But this was shown in Lemma
\ref{Lfilt} above.

Now, since the generating central element of $H$ lies in $HV^2 (\subset
[H,V])$, it must act trivially on $H/[H,H]$, which concludes the proof.
\end{proof}


\begin{cor}\label{C2}
For any $z,\ \mf{Z}(\mf{Ug})\cong \mf{Ug} / [\mf{Ug}, \mf{Ug}]$ surjects
onto $H_z / [H_z, H_z]$. Every $X \in F^n \mf{Ug}$ is equivalent to some
$X' \in F^n \mf{Ug} \cap \mf{Z}(\mf{Ug})$ modulo $[\mf{Ug}, \mf{Ug}]$ or
$[H_z, H_z]$.
\end{cor}

\begin{proof}
We make many statements here. The first equality comes from the fact that
$\mf{Ug}$ is a direct sum of finite-dimensional $\mf{sl}_2$-modules
(e.g., by Lemma \ref{L0} with $V = \mf{h}' = 0$), whence the images of
$\ad e$ and $\ad f$ span a complement to the center (using weight
vectors). Moreover, no polynomial in the Casimir is in the commutator,
since one can always find a finite-dimensional $\mf{Ug}$-module on which
it has nonzero trace.

Now for the surjection: we first claim that $\mf{Ug}$ surjects onto the
abelianization of $H_z$. Indeed, the main step in showing this is the
$z=0$ case, which is the proposition above: $\mf{Ug} \twoheadrightarrow
\mf{Z}(\mf{Ug}) \mapdef{\sim} H / [H,H]$. But this implies that $\mf{Ug}$
surjects onto the associated graded of the abelianization of $H_z$, since
$H / [H,H] \twoheadrightarrow \Gr(H_z / [H_z,H_z])$.

So we just need to show that this can be ``lifted" to a surjection as
desired. Now given $a \in F^n H_z$ (for the usual filtration on $H_z$),
we can find $c \in \mf{Ug}$ and $a_i, b_i \in H_z$ such that the
filtration degrees of $a_i, b_i$ always add up to at most $n$, and
$\overline{a} = \overline{c} + \sum_i [\overline{a_i}, \overline{b_i}]$
in the associated graded, from above. But then $a - c - \sum_i [a_i, b_i]
\in F^{n-1} H_z$, and we can proceed by induction.

Finally, $\mf{Ug} \hookrightarrow H_z$, so $[\mf{Ug}, \mf{Ug}]$ is killed
by the map : $\mf{Ug} \twoheadrightarrow H_z / [H_z, H_z]$. Hence
$\mf{Z}(\mf{Ug})$ surjects onto $H_z / [H_z, H_z]$.

Next, we show the last statement. Consider the finite-dimensional
$\mf{Ug}$-submodule $M_n := F^n \mf{Ug} \subset \mf{Ug}$, and its
submodule $[\mf{g}, M_n] \subset M_n$. Clearly, $M_n / [\mf{g}, M_n]$
surjects onto the image of $M_n$ modulo $[\mf{Ug}, \mf{Ug}]$ or $[H_z,
H_z]$; on the other hand, $M_n / [\mf{g}, M_n]$ is isomorphic to
$\mf{Z}(\mf{Ug}) \cap M_n$ by complete reducibility. We are done.
\end{proof}

\noindent Thus we need to compute the kernel (which obviously contains at
least $z$).\medskip

As an aside, we note that equation \eqref{Ecoh} holds for general $z$:

\begin{lemma}
$H_z/[H_z,H_z]=(H_z/[H_z,V])^{\mf{g}}$.
\end{lemma}

\begin{proof} 
Consider the following sequence of $H_z$-bimodules:
\[ H_z \otimes (V \oplus \mf{g}) \otimes H_z \to H_z\otimes H_z \to H_z
\to 0 \]

\noindent where the last map is the multiplication map, and the first map
is given by $w \mapsto 1 \otimes w - w \otimes 1$ for any $w \in V \oplus
\mf{g}$. We claim that this sequence is right exact. Indeed, we only need
to verify exactness of the middle term. But all terms of this sequence
are naturally filtered, and after passing to the associated graded
picture, we will get an analogous sequence for $H = \mf{U}(\mf{g} \ltimes
V)$, for which the sequence is well known to be exact. But since
$H_z/[H_z,H_z] = Tor_0(H_z,H_z)$ in the category of $H_z$-bimodules,
after tensoring our sequence with $H_z$ we get that
\[ H_z/[H_z,H_z] = H_z/[H_z,V \oplus \mf{g}] = (H_z/[H_z,V])^{\mf{g}} \]

\noindent and we are done.
\end{proof}\medskip

We finally have the following theorem.

\begin{theorem}\label{Tzz}
Let the parameter $z$ be nonzero, say $z = c \Delta^m + l.o.t.$. If $m=0$
(i.e., $z$ is a constant), then the commutator quotient of $H_z$ is
trivial. Otherwise, if $\deg z = m \geq 1$, then $1, \Delta, ...,
\Delta^{m-1}$ are linearly independent in $H_z/[H_z,H_z]$, and generate
it as a module over the center of $H_z$.
(In particular, $(H_z/[H_z,H_z])/(t_z)$ is a vector space of dimension
$m$ over $k$.)
\end{theorem}

An important first step in showing this, is the following proposition.

\begin{prop}\label{Pstep}
For all $a,b \geq 0$, $t_z^a \Delta^b$ equals a (nonzero) polynomial in
$\Delta$ of degree $a(m+1) + b$, modulo $[H_z, H_z]$.
\end{prop}

\begin{proof}
The case $a=0$ is obvious; we will show the $a=1$ case below. The case of
higher $a$ is then proved by induction on $a$: for a fixed $b$, if $t_z^a
\Delta^b - p_{ab}(\Delta) = \sum_i [r_i, s_i] \in [H_z, H_z]$, then
\[ t_z^{a+1} \Delta^b = t_z p_{ab}(\Delta) + t_z \sum_i [r_i, s_i] = t_z
p_{ab}(\Delta) + \sum_i [t_z r_i, s_i] \]

\noindent and $t_z p_{ab}(\Delta)$ can be rewritten appropriately, using
the $a=1$ statement (for various $b$).

It remains to show the hypothesis for $a=1$ and all $b$. In the rest of
the proof, we will use the following result several times.

\begin{lemma}\label{Lstar}\hfill
\begin{enumerate}
\item Let $d = [\alpha,x] + [\beta,y]$, with $\alpha, \beta \in \mf{Ug}$.
Then modulo $[H_z, H_z],\ dx \equiv -\beta z,\ dy \equiv \alpha z$.

\item For any $z' \in \mf{Z}(\mf{Ug}),\ z'ey^2 \equiv z' hxy \equiv -z'
fx^2 \mod [H_z, H_z]$.
\end{enumerate}
\end{lemma}

\begin{proof}\hfill
\begin{enumerate}
\item We have $dx = [\alpha x, x] + [\beta x, y] - \beta z$, and $dy =
[\alpha y, x] + [\beta y, y] + \alpha z$. Both claims now follow.

\item Note that $[f, z'exy] = -z'hxy + z'ey^2$, which proves the first
equality; for the second, apply the anti-involution $j$. We note that $j$
fixes the Casimir element, and hence the whole center. Applying $j$ to
the above equation, $-xyhz' - x^2 fz' = A \in [H_z, H_z]$, say. Hence we
make the following reductions:
\begin{eqnarray*}
-z'fx^2 & = & -fz'x^2 = -x^2 fz' + [x^2, fz'] = [x^2, fz'] + A + xyhz',\\
xyhz' & = & z' xyh + [xyh, z'] = z'hxy + [xyh, z']\ (\mbox{since } xy
\mbox{ has weight } 0).
\end{eqnarray*}

Thus, $-z'fx^2 \equiv z'hxy \mod [H_z, H_z]$, as claimed.
\end{enumerate}
\end{proof}

We now prove the result for $t_z \Delta^n$ ($n \geq 0$). Since $t_z =
(ey^2 + hxy - fx^2) - \half hz - q_z$ (see equation \eqref{Eqz}), and
since $hz = [e,fz] \in [\mf{g},H_z]$, we have
\begin{equation}\label{Ezz}
t_z \Delta^n \equiv \Delta^n t_z \equiv \Delta^n (3 hxy - q_z) \mod [H_z,
H_z].
\end{equation}

By Lemma \ref{Lfilt}, $\Delta^n hx \in \mf{Ug} \cdot V = [\mf{Ug}, V]$ is
of the form $[a_n, x] + [b_n, y]$ for some $a_n, b_n \in \mf{Ug}$. By
Corollary \ref{C2}, we may assume that $a_n, b_n \in \mf{Z}(\mf{Ug}) \cap
F^{2n+1} \mf{Ug}$ (modulo the commutator). By Lemmma \ref{Lstar}, $3
\Delta^n hxy \equiv 3 a_n z \mod [H_z, H_z]$.

We thus have to prove (using equation \eqref{Ezz}) that $3 a_n z -
\Delta^n q_z$ is a polynomial of degree $n+m+1$ in $\Delta$. In light of
Corollary \ref{C3}, it suffices to show that $a_n$ is a polynomial of
degree $n+1$ with positive (rational) top coefficient (in fact, it turns
out to be $1/6(n+1)$).\medskip

To do this, consider the formula for $[\Delta^n, x]$, which yields: $f_n
\cdot (hx + 2ey) = [\Delta^n, x] - g_n x$. Again using Lemma \ref{Lfilt}
and Corollary \ref{C2}, write $g_n x = [c_n, x] + [c'_n, y]$ for $c_n,
c'_n$ polynomials in $\Delta$. Moreover, since $\deg(g_n(T)) = n-1$,
$c_n, c'_n \in F^{2n-1}  \mf{Ug}$; thus, $\deg(c_n) < n$ (as a polynomial
in $\Delta$).

But then Lemma \ref{Lstar} implies that on the one hand,
\[ f_n (hx + 2ey)y \equiv f_n (hxy + 2ey^2) \equiv f_n (3hxy) \mod [H_z,
H_z] \]

\noindent and on the other (modulo the commutator),
\[ f_n (hx + 2ey)y \equiv (\Delta^n - c_n)z \equiv c \Delta^{m+n} +
l.o.t.. \]

We thus get: $f_n (3hxy) \equiv c \Delta^{m+n} + l.o.t.$ for all $n$.
Using the ``unipotent" (with positive coefficient $1/(2n)$) change of
basis from $f_n$ to $\Delta^n$,we get
\[ \Delta^n (3hxy) = \left( \frac{1}{2n+2} f_{n+1} + ``l.o.t." \right)
(3hxy) = \frac{c}{2n+2} \Delta^{m+n+1} + l.o.t. \]

\noindent where $``l.o.t."$ stands for ``lower-degree" $f_i$'s. Now
compare this to what we had above:
\[ \Delta^n (3hxy) \equiv 3 a_n z = a_n (3c \Delta^m + l.o.t.), \]

\noindent and we are done.
\end{proof}

\begin{proof}[Proof of Theorem \ref{Tzz}]
First of all we have $[H_z,H_z] \cap \mf{Ug} \subseteq z \cdot \mf{Ug} +
[\mf{Ug},\mf{Ug}]$ (since any time the filtration degree in $x,y$ goes
down in a commutator expression, a multiple of $z$ appears). Since
$\mf{Ug} / (\mf{Ug} \cap [H_z, H_z]) \subseteq H_z / [H_z, H_z]$, and
$\mf{Ug} / (\mf{Ug} \cap [H_z, H_z])$ surjects onto
$\mf{Ug} /(z \cdot \mf{Ug} + [\mf{Ug},\mf{Ug}]) = \mf{Z}(\mf{Ug}) / z
\cdot \mf{Z}(\mf{Ug})$,
hence the elements $1,...,\Delta^{m-1}$ are linearly independent in $H_z
/ [H_z, H_z]$.

It remains to show that the following elements span $H_z / [H_z, H_z]$ -
or in light of Corollary \ref{C2}, the center of $\mf{Ug}$: $\{ t_z^a
\Delta^b : a \geq 0,\ 0 \leq b < m \}$. We now show that all $\Delta^n$
lie in this span, modulo $[H_z, H_z]$. Clearly, $1, \dots, \Delta^{m-1}$
as well as $\Delta^m = (1/c)z - l.o.t.$ are in this span, since $z/c =
[x/c,y]$. Next, $\Delta^{m+1}, \dots, \Delta^{2m}$ are in the span: just
consider $t_z, t_z \Delta, \dots, t_z \Delta^{m-1}$. As for
$\Delta^{2m+1}$, we have
\[ \Delta^{2m+1} \equiv t_z z + l.o.t. \equiv [x/c, t_z y] + l.o.t. \mod
[H_z, H_z] \]

\noindent similar to above. Keep repeating this procedure.
\end{proof}

We expect that a stronger statement is true: namely, that the commutator
quotient is actually a free module over the center, with basis $1,
\Delta, \dots, \Delta^{m-1}$. This would imply (via Hochschild cohomology
considerations) that the algebras $H_{z_1}/(t_{z_1}-a),
H_{z_2}/(t_{z_2}-b)$ are not Morita equivalent if $\deg z_1 \neq \deg
z_2$, where $a,b\in k$.\medskip

\section{Infinitesimal Hecke algebra of $\mf{gl}_n$}
 
We now recall the definition of an infinitesimal Hecke algebra of
$\mf{g}=\mf{gl}_n$ and $V=\mf{h}\oplus \mf{h}^{*}$, where $\mf{h} = k^n$
and $\mf{h}^*$ is its dual representation. We (again) identify $\mf{g}$
with $\mf{g}^*$ via the pairing $\mf{g} \times \mf{g} \to k :\ (A,B)
\mapsto \tr(AB)$, and identify $\mf{Ug}$ with $\Sym \mf{g}$ via the
symmetrization map.

Then for any $x \in \mf{h}^{*},\ y \in \mf{h},\ A \in \mf{g}$, one writes 
\[ (x, (1 - T A)^{-1} y) \det (1 - T A)^{-1} = r_0(x,y)(A) + r_1(x,y)(A)
T + \dots \]

\noindent where $r_i(x,y)$ is a polynomial function on $\mf{g}$, for all
$i$.

Now for each polynomial $\beta = \beta_0 + \beta_1 T + \beta_2 T^2 +
\dots \in k[T]$, the authors define in \cite{EGG} the algebra $H_\beta$
as a quotient of $T(\mf{h} \oplus \mf{h}^*) \rtimes \mf{Ug}$ by the
relations
\[ [x,x'] = 0,\qquad [y,y'] = 0,\qquad [y,x] = \beta_0 r_0(x,y) + \beta_1
r_1(x,y) + \dots \]

\noindent for all $x,x'\in\mf{h}^{*},\ y,y'\in\mf{h}$. It is proved in
\cite{EGG} that these algebras are infinitesimal Hecke algebras. Also
note that if $\beta \equiv 0$, then $H_0 = \mf{U}(\mf{gl}_n \ltimes
(\mf{h} \oplus \mf{h}^*))$.

\subsection{Relations and anti-involution}

We start with an explicit presentation of $H_\beta$: it is generated by
$\mf{gl}_n = \bigoplus_{i,j} k e_{ij}$ and $\mf{h} = \bigoplus_i k v_i,\
\mf{h}^* = \bigoplus_i k v_i^*$, where $\{ v_i \}, \{ v_i^* \}$ form dual
bases of $\mf{h}, \mf{h}^*$ respectively. We have the relations:
\[ e_{ij} \cdot v_k := \delta_{jk} v_i, \qquad e_{ij} \cdot v_k^* := -
\delta_{ik} v_j^*, \qquad v_i^*(v_j) = \delta_{ij}. \]

We next describe an anti-involution of $H_\beta$, for (at most) linear
$\beta$. Suppose we have $j$ sending $e_\alpha \leftrightarrow f_\alpha$
and $h \leftrightarrow h$ for all positive simple roots $\alpha$ for a
reductive Lie algebra $\mf{g}$ (and Cartan subalgebra elements $h$). One
then checks that this gives an anti-involution $j$ of $\mf{g}$ (and hence
of $\mf{Ug}$).

Now let $\mf{g} = \mf{gl}_n$; then $j(X) = X^T$ in $\mf{g}$. We now
mention the anti-involution.

\begin{lemma}
The map $j : (X, v_i) \leftrightarrow (X^T, -v_i^*)$ extends to an
anti-involution of $\mf{Ug} \ltimes T(\mf{h} \oplus \mf{h}^*)$.
Moreover, $j$ factors to an anti-involution of $H_\beta$ when $\beta$ is
at most linear.
\end{lemma}

\begin{proof}
For the first part, we only need to check that $j$ preserves (actually,
permutes) the following relations:
\[ [e_{ij}, e_{kl}] = \delta_{kj} e_{il} - \delta_{il} e_{kj}, \quad [
e_{ij}, v_k] = \delta_{jk} v_i, \quad [e_{ji}, v_k^*] = - \delta_{jk}
v_i^*\ \forall i,j,k,l. \]

\noindent This is easy to do. Next, for $H_\beta$ with $\beta$ at most
linear, we refer to \cite[Examples 4.6, 4.7]{EGG}; thus, $H_\beta$ is the
quotient of the above algebra, by the relations
\[ [v_i, v_j] = [v_i^*, v_j^*] = 0, \qquad
[v_i, v_j^*] = \delta_{ij} (\beta_0 + \beta_1 \tau) + \beta_1 e_{ij}\
\forall i,j, \]

\noindent where $\tau = \Id_n \in \mf{gl}_n$. That $j$ preserves these
relations, is is also easy to verify.
\end{proof}

\subsection{Central elements}

We now mention discuss central elements for various $\beta$ (and general
$n$). We first have a result for $\beta \equiv 0$, which can be verified
using a strategy similar to the proof of Proposition \ref{Psp}.

\begin{prop}
The center of $H_0(\mf{gl}_n)$ contains at least two algebraically
independent elements, both fixed by $j$:
\[ r_n := \sum_{i=1}^n v_i v_i^*,\ s_n := \sum_{1 \leq p < q \leq n}
(e_{pq} v_q v_p^* + e_{qp} v_p v_q^*) - (e_{pp} v_q v_q^* + e_{qq} v_p
v_p^*). \qed \]
\end{prop}

Next, we prove that in general, $H_\beta$ (over $\mf{gl}_n$) has
nontrivial center, by providing a lift $r_\beta$ of $r_n$; clearly,
$r_\beta$ is transcendental in $H_\beta$ since $r_n$ is thus in $H_0$.

\begin{prop}\label{Pgln}
For any $n,\beta$, $H_\beta$ contains the central element $r_\beta := \h
+ \tau$ (which is transcendental in $H_\beta$).
\end{prop}

\noindent Here, $\tau = \Id_n$, and $\h$ is the {\it Euler element} in
\cite[\S 5.2]{EGG}, given by
\[ \h = \sum_i v_i^* v_i + \frac{n}{2} + c, \]

\noindent where $c \in \bggo(G)^*$ is defined via the following equation
(see \cite[\S 3.4]{EGG}), with $t \in k$:
\[ \kappa(x,y) := [x,y] = (y,x)t + (y, (1-g)x)c, \mbox{ for all } x \in
\mf{h}^*, y \in \mf{h}. \]

\begin{proof}
(Note that $k$ is algebraically closed, of characteristic zero.) As
mentioned in \cite[\S 4.1]{EGG}, the infinitesimal Hecke algebra
$H_\beta$ only exists when $\im(\kappa) \subset \mf{Ug}$; thus, $f_{xy}
\cdot c \in \mf{Ug}$ for all $f_{xy} := (y, (1-g)x) \in \bggo(G)$ (with
$x \in \mf{h}^*, y \in \mf{h}$).
By the Nullstellensatz, $c \in \mf{Ug}$, so $\h \in \sum_i v_i v_i^* +
\mf{Ug}$ now; therefore $r_\beta$ is indeed a lift of $r_n$ to $H_\beta$.
That it is central follows from \cite[Proposition 5.3]{EGG}, and because
$\h, \tau$ commute with $\mf{gl}_n$.
\end{proof}

\section{Category $\bggo$ for Infinitesimal Hecke algebras}

At first, let us discuss an analogue of the BGG category $\bggo$, for a
class of algebras equipped with the following structure:\medskip

Let $A\supset k$ be an associative  algebra, endowed with the following
additional structure:
\begin{itemize}
\item $A$ has an increasing filtration by $k$-subspaces  $F^n A, n \geq
0$, that satisfy $F^n A \cdot F^m A \subseteq F^{n+m}A$; 

\item There are three finite-dimensional $k$-subspaces $\mfn^+, \mfn^-,
\mfh \subseteq F^1 A$, such that $\mfn^+ + \mfn^- + \mfh = \mfn^+ \oplus
\mfn^-\oplus \mfh$.
\end{itemize}\medskip

\noindent From these data we require that
\begin{itemize}
\item $A$ is generated as an algebra over $k$ by $\mfn^+\oplus
\mfn^-\oplus \mfh$; each summand is a Lie (sub)algebra, and
\[ [\mfh,\mfh]=0,\ [\mfh, \mfn^+] = \mfn^+,\ [\mfh,\mfn^-] = \mfn^-. \]

\item There is a (fixed) subspace $\mfh_0 \subset \mfh$, and both
$\mfn^+$ and $\mfn^-$ are diagonally acted upon by the adjoint action of
$\mfh$, and the eigenvalues occurring in these decompositions have images
in opposite non-intersecting cones in $\mfh_0^* (\twoheadleftarrow
\mfh^*)$.

\item The multiplication map $: B_1 \otimes B_2 \otimes B_3 \to
\Gr(F^\bullet A)$ is a vector space isomorphism, where $\{ B_1, B_2, B_3
\} = \{ \mf{Un}^-, \mf{Un}^+, \Sym(\mfh) \}$ (i.e., in every possible
order). Moreover, $\Sym \mfh \subset F^0 A$.

\item In addition, we require that $\Gr(F^\bullet A)$ is equipped with a
filtration consisting of finite-dimensional subspaces $G^n\ (n\geq 0)$,
such that $\mfn^+ \oplus \mfn^- \oplus \mfh \oplus k = G^1 \Gr(A)$, and
$\Gr(\Gr(F^\bullet A))$ is a polynomial algebra, i.e., $\Sym(\mfn^+
\oplus \mfn^- \oplus \mfh) \to \Gr(\Gr(F^\bullet A))$ is an isomorphism.
\end{itemize}

\noindent Moreover, if $A$ is such an algebra, then so are $\Gr(F^\bullet
A)$ and $\Gr(G^\bullet(\Gr(F^\bullet A)))$.\medskip

Of course, the main examples we have in mind are infinitesimal Hecke
algebras (the axiomatics of category $\bggo$ in more general settings is
considered in \cite{Kh2}). The axiom about $\mfh_0 \subset \mfh$ is
needed (later) for technical purposes: though we can choose $\mfh_0 =
\mfh$ for $H_z$ (over $\mf{sl}_2$), we need to choose $\mfh_0 = k h
\subset \mfh = k h \oplus k \tau$ in $H_\beta$ (for $\mf{gl}_2$).
Moreover, for infinitesimal Hecke algebras, we clearly have
$\Gr(F^\bullet H_\beta) = H_0 = \mf{U}(\mf{g} \ltimes V)$ and
$\Gr(G^\bullet H_0) = \Sym(\mf{g} \oplus V)$.

We now mimic some standard definitions.

\begin{defin}\hfill
\begin{enumerate}
\item The {\it category} $\bggo$ for the algebra $A$ (as above), denoted
by $\bggo_A$, is the full subcategory of finitely generated left
$A$-modules, defined by: $M \in \bggo_A$ if and only if $\mfn^+$ acts
locally nilpotently on $M$, and $\mfh$ acts on it diagonalizably with
finite-dimensional eigenspaces. That is, $M = \bigoplus_{\chi \in \mfh^*}
M^{\chi}$, with $\dim M^{\chi} < \infty\ \forall \chi$.

\item An element $v \in M$ is said to be a {\it maximal vector} if it is
an eigenvector for the $\mfh$-action, and $\mfn^+ v = 0$.

\item ({\bf Definition-proposition}.) Let $\chi \in \mfh^*$. Then there
exists an object $M(\chi) \in \bggo_A$, characterized by the following
uniqueness property: $M(\chi)^{\chi}=k$, and if $v\in M(\chi)^{\chi}$,
then for every pair $v_1, M_1$ with $v_1 \in M_1^\chi$ a maximal vector,
there exists a unique $f \in \Hom_A(M(\chi), M_1)$ such that $f(v) =
v_1$. Such a module $M(\chi)$ is called a {\it Verma module} for the
weight $\chi$.

\begin{proof}
Let $A_{-}$ be the subalgebra of $A$ generated by $\mfh \oplus \mfn^-$;
then there exists $\chi:A_{-}\to k$, such that $\chi|_{\mfh} = \chi,\
\chi(\mfn^-)=0$.
Indeed, we just need to check that $\mfh \cap \mfn^- A_- = 0$, which is
immediate from weight space theory. But then, $k$ turns into a left
$A_-$-module (which will be denoted by $k_\chi$).

Now define $M(\chi) := A \otimes _{A_-} k_\chi$. It is clear that this
module lies in $\bggo_A$ and $v = 1 \otimes 1$ is a maximal nonzero
vector of weight $\chi$. If $v_1\in M^{\chi}$ is a maximal vector in an
$A$-module ($M$), then we have a map of $A_-$-modules $f : k_{\chi} \to
M$ such that $f(1) = v_1$.
Hence we get $f \otimes_{A_-} \Id : A \otimes_{A_-} k_\chi \to A
\otimes_{A_-} M \to M$, such that $v$ maps to $v_1$; obviously this map
is unique.
\end{proof}
\end{enumerate}
\end{defin}\medskip

We have the following standard

\begin{prop}
For any $\chi \in \mfh^*,\ M(\chi)$ has a unique maximal subobject and
irreducible quotient (both in $\bggo_A$); call the latter $V(\chi)$. Then
every irreducible object in $\bggo_A$ is of the form $V(\chi)$ for some
$\chi \in \mfh^*$.
\end{prop}

\begin{proof}
If $V\subset M(\chi)$ is a proper subobject, then $V^\chi = 0$. Hence the
sum of all proper subobjects of $M(\chi)$ is still a proper submodule,
which proves the first assertion. Now if $V$ is an irreducible object,
then it must have a maximal vector $v \in V^\chi$ for some $\chi$. Hence
$\Hom(M(\chi), V) \neq 0$, so $V = V(\chi)$.
\end{proof}
   
As is usual in representation theory, one would like to study
(irreducible) finite-dimensional representations, compute the
multiplicity of $V(\chi)$ in $M(\mu)$ (for all $\chi, \mu \in \mfh^*$),
and so on. One has the usual spectral decomposition of $\bggo_A$ with
respect to its center: $\bggo_A = \bigoplus_{\phi \in \spec (\mf{Z}(A))}
\bggo^\phi$, where $\bggo^\phi$ is the full subcategory consisting of
objects on which $\phi(t)-t$ acts locally nilpotently for any $t \in
\mf{Z}(A)$. In particular, we have a Harish-Chandra map $\eta : \mfh^*
\to \spec \mf{Z}(A)$.

Let us compare $\bggo_A$ and $\bggo_{\Gr(A)}$. If $M \in \bggo_A$, let $V
\subset M$ be a finite-dimensional vector space generating $M$ over $A$.
Then $M$ has the usual increasing filtration: $F^n M := (F^n A)V$, which
makes $\Gr(M)$ a $\Gr(A)$-module (note that this construction depends on
our choice of $V$).

Moreover, $\Gr(M)$ belongs to $\bggo_{\Gr(A)}$, and $\Ch_{\bggo_A}(M) =
\Ch(\Gr(M))$ (where $\Ch(M) := \sum_{\chi\in \mfh^*} (\dim
(M^{\chi})\chi)$ is the character of an $\mfh$-semisimple module). Hence
we see that $\bggo_{\Gr(A)}$ provides an ``upper bound" for $\bggo_A$
(i.e., $\Gr(M) \in \bggo_{\Gr(A)}\ \forall M \in \bggo_A$).

We also remark that if we start with a Verma module $M(\lambda)$ and $V =
k \cdot v_\lambda$ (the highest weight space in it), then we will get a
Verma module $\Gr(M(\lambda))$ over $\Gr(A)$ of weight $\lambda$. In
particular,
\begin{equation}\label{Eann}
\Gr(\Ann(M(\lambda))) \subseteq \Ann(\Gr(M(\lambda))).
\end{equation}

\noindent This fact is used in the section about primitive
ideals.\medskip

In the remaining part of this section, we focus on the category $\bggo$
for $A = H_z$ (which does fit into the above setup). This category was
studied in great detail in \cite{Kh}. We now reinterpret some of those
results using the center of $H_z$. We have $\bggo = \bigoplus_{\lambda
\in k} \bggo^\lambda$, where $(t_z - \lambda)$ acts nilpotently on
$\bggo^\lambda$ (though as we see presently, $\bggo = \bggo^0$ if $z =
0$).

At first, let us compute the action of $t_z$ on $M(\lambda)$. We have
\begin{eqnarray*}
&& (ey^2+hxy-fx^2 - \half hz - q_z) v_{\lambda}\\
& = & (y^2e + 2yx + z + hyx + hz - \half hz - q_z) v_{\lambda} = ((1
+ \half h) z - q_z) v_{\lambda}\\
& = & \left(\half \lambda +1 \right) z(\lambda^2 + 2\lambda) v_\lambda -
q_z (\lambda^2 + 2 \lambda)v_{\lambda}.
\end{eqnarray*}

\noindent Let us denote by $\phi_z(t)$ the following polynomial in
$k[t]$:
\begin{equation}\label{Ecenter}
\phi_z(t) = \left( \half t + 1 \right) z(t^2 + 2t) - q_z(t^2 + 2t),
\end{equation}

\noindent where as usual, we treat $z$ as a polynomial of $\Delta$ (note
that $\phi_0(t) \equiv 0$). As a corollary, $V(\lambda)\in \bggo^{\mu}$
only if $\phi_z(\lambda) = \mu$.

Now suppose $z \neq 0$. Then the degree of $\phi_z(t)$ equals $2(\deg(z)
+ 1)$, and the multiplicity of $V(\lambda)$ in $M$ is  at most $\dim_k
M^{\lambda}$. Hence all Verma modules - and thus, all objects in category
$\bggo$ - have finite length.

Moreover, every central character of $H_z$ is of the form $\chi_\mu : t_z
\mapsto \mu \in k$, and since $k$ is algebraically closed, and $\deg
\phi_z > 0$, we can find $\lambda \in k$ such that $\phi_z(\lambda) =
\mu$.
To summarize, we get the following result, most of which is contained in
\cite{Kh}, but is proved there by a completely different approach.

\begin{prop}
Each module in $\bggo^\lambda$ (for any $\lambda$) has finite length, and
$V(\mu) \in \bggo^\lambda$ if and only if $\mu \in \phi_z^{-1}(\lambda)$.
In particular, the number of non-isomorphic irreducible objects in
$\bggo^\lambda$ is at most $2(\deg(z)+1)$. Furthermore, every central
character for $H_z$ is associated to some Verma module.
\end{prop}

As an aside, the algebra $H_z$ has the following peculiar property:

\begin{prop}
If the parameter $z$ is nonzero, then there are at most finitely many
non-isomorphic irreducible finite-dimensional $H_z$-modules.
\end{prop}

\begin{proof}
For the proof, we are going to use a theorem proved by Khare in
\cite{Kh}. We need to recall some definitions from there. For any pair of
integers $r,m$, he considers the following expression:
\[ \alpha_{rm}=\sum_{i=0}^{m-2} (r+1-i)(z(r+1-i)^2-1) \]

\noindent (where $z(-)$ is viewed as a polynomial in the Casimir
element). Then his result (\cite[Theorem 11]{Kh}) says that $V(r)$ is
finite-dimensional if and only if there exists a nonnegative integer $s
\leq r$ such that $\alpha_{r,r-s+2}=0$.

Let us explain why this can not happen as long as $z \neq 0$ and $r$ is
large enough. We may rewrite $\alpha_{rm}$ as follows:
\[ \alpha_{rm}=\sum_{i=1}^{r+1} i z(i^2-1)-\sum_{i=1}^{r+2-m} i z(i^2-1),
\]

\noindent Therefore if we denote $\sum_{i=1}^j i z(i^2 - 1)$ by $f(j)$
(thus $f$ is a polynomial of some positive degree), then $\alpha_{rm} =
f(r+1) - f(r+2-m)$. So if $V(r)$ is finite-dimensional, then $f(r+1) =
f(r+2 - (r-s+2)) = f(s)$ for some $0 \leq s \leq r$. It thus suffices to
show that for a nonconstant polynomial $f \in k[T]$, the numbers $f(1),
f(2), \dots$ are ``eventually pairwise distinct"; we show this now, in
Lemma \ref{Lpoly}.
\end{proof}

\begin{lemma}\label{Lpoly}
Suppose $f \in k[T]$ is a nonconstant polynomial with coefficients in a
field of characteristic zero. Then beyond some $r_0 \gg 0$ (in $\Q
\hookrightarrow k$), $f : [r_0, \infty) \cap \Q \to k$ is injective.
\end{lemma}

\noindent This result does not generalize (much) more; consider $f(T) =
T^2$ evaluated at $0, 1, -1, 2, -2, \dots$ in $\Q$.

\begin{proof}
Consider the coefficients $c_0, \dots, c_d \in k$ of $f(T) = c_0 + c_1 T
+ \dots + c_d T^d$. Now choose any $\Q$-basis $\{ b_1, \dots, b_s \}$ of
the $\Q$-span of the $c_i$'s, and rewrite $f(T) = f_1(T) b_1 + \dots +
f_s(T) b_s$, where $f_i(T) \in \Q[T]$. Then at least one polynomial is
nonconstant, say $f_1$ (without loss of generality).

Now, the absolute value of $f_1(r)\ (r \in \Q)$ is a strictly increasing
function of $r$ for $r \gg 0$, and this proves the result (since the
$b_i$'s are $\Q$-linearly independent).
\end{proof}

\section{Primitive ideals of $H_z$}        

Let us start with the following definition.

\begin{defin}
We say that a (unital) $k$-algebra $A$ is {\it almost commutative (of
order 1)} if it admits an increasing filtration $F^\bullet A$ such that
the corresponding associated graded is a finitely generated commutative
$k$-algebra.

For $n>1$, we say that a $k$-algebra is {\it almost commutative of order
$n$} if it admits an increasing filtration compatible with the algebra
structure, such that the associated graded is an almost commutative
algebra of order $n-1$.
\end{defin}

We have the following direct generalization of Quillen's theorem
\cite{Q}, whose proof goes through essentially word by word; we reproduce
this proof for the reader's convenience. (In what follows, $k$ is an
arbitrary field.)\medskip

\begin{theorem}[Quillen]
Let $A$ be an almost commutative algebra of some order and let $M$ be a
simple module over $A$. If $\phi\in \End_A(M)$, then $\phi$ is algebraic
over $k$.
\end{theorem}

\begin{proof}
Note the following elementary facts: if a $k$-algebra $B$ is filtered
with associated graded algebra $C = \Gr(F^\bullet B)$, then any finitely
generated $B$-module $M$ is automatically filtered as well: let $V$ be
the $k$-span of a (finite) set of generators for $M$, and define a
filtration on $M$ via:
\[ F^i M = F^i B \cdot V. \]

\noindent Then $\Gr F^\bullet M$ is automatically a finitely generated
$C$-module. Moreover, $\Gr(B[T]) = C[T]$. Finally, choose $\theta \in
\End_B M$; then $M$ is naturally a $B[T]$-module, via: $(b \otimes
p(T))(m) := p(\theta)(b \cdot m) = b \cdot p(\theta)(m)$. Then $\Gr
F^\bullet M$ is a finitely generated module over $C[T]$ (as mentioned in
\cite{Q}; here, $T \mapsto \Gr(\phi)$).\medskip

We now ``rewrite" the proof from \cite{Q}. Note that $M$ is an
$A[T]$-module as above (with $T \mapsto \phi$); taking the associated
graded of this (successively), we get a finitely generated module $N$
over $B[T]$, where $B$ is almost commutative, and $N$ is obtained from
$M$ by taking successive associated graded modules in a standard way.
Then $\Gr(N)$ is finitely generated over $\Gr(B[T])$.

By the generic flatness lemma (see \cite{Q}), there exists a nonzero
polynomial $f \in k[T]$, such that $\Gr(N)$ is free over $k[T]_f$. This
implies that $N$ is free over $k[T]_f$, whence we will get that so is $M$
(with $T \mapsto \phi$ when acting on $M$). On the other hand,
$\End_A(M)$ is a skew field, so $M$ is a vector space over $k(\phi)
\subset \End_A(M)$. This is a contradiction if $\phi$ is transcendental
over $k$.
\end{proof}

Next, recall the following definition from \cite{Gi}.

\begin{defin}
Let $k \subset A$ be an associative algebra endowed with two (non-unital)
finitely generated commutative subalgebras $A'_\pm$ and an element
$\delta \in A$. One says that this data defines an {\it algebra with
commutative triangular decomposition} if the following hold:
\begin{itemize}
\item $\ad \delta $ preserves both $A'_\pm$;

\item $\ad \delta$ acts diagonalizably on $A$; the eigenvalues for the
action on $A'_\pm$ lie in $\pm \Z_{>0}$; and

\item the algebra $A$ is finitely generated as an $A_-$-$A_+$ bimodule,
where $A_\pm := A'_\pm \oplus k \subset A$. (This differs from \cite{Gi}
in order to reconcile our notion of $\bggo$ to his.)
\end{itemize}
\end{defin}

\noindent In this case, Ginzburg's ``Generalized Duflo Theorem"
\cite[Theorem 2.3]{Gi} (which actually concerns a wider class of
algebras) says that primitive ideals are the same as prime ideals, and
are annihilators of simple objects of the appropriately defined BGG
category $\bggo$ (provided it has finitely many simple objects). Applying
this to our algebra $H_z$, we get:

\begin{theorem}[Analogue of Duflo's theorem]
Primitive ideals in $H_z$ are the same as prime ideals, and are
annihilators of simple objects in $\bggo$.
\end{theorem}

\begin{proof}
Let $R_\lambda := H_z / (t_z - \lambda) H_z$. Given a primitive ideal $I
\subset H_z$, we get a simple $H_z$-module $M$; since $k = \overline{k}$,
Quillen's theorem says that $M$ is a simple $R_\lambda$-module for some
$\lambda \in k$.

Suppose we show that $A = R_\lambda$ is a finitely generated
$A_-$-$A_+$-bimodule, where $A_\pm$ are the images of $B_+ := k[e,x], B_-
:= k[f,y]$ (respectively) under the quotient map $(a \mapsto
\overline{a}) : H_z \twoheadrightarrow R_\lambda$. Then Ginzburg's
theorem holds for $R_\lambda$ (using $\delta = \overline{h}$ and $A'_\pm$
to be the augmentation ideals in $A_\pm$). Moreover, the category
$\bggo_{R_\lambda}$ is contained in $\bggo_{H_z}^\lambda$, the summand in
the spectral decomposition mentioned in a previous section, and hence it
contains only finitely many simples.

Thus, primitive ideals for $H_z$ are indeed annihilators of simple
objects in $\bggo_{H_z}$. Moreover, $\overline{I}$ is prime, hence so is
$I$. Conversely, if $I$ is prime, then so is $\overline{I}$, whence it
annihilates a simple object in $\bggo_{R_\lambda}$. Thus, $I$ annihilates
some $V(\mu) \in \bggo_{H_z}$.\medskip

Therefore, it suffices to show that $H_z / (t_z - \lambda) H_z$ is
finitely generated as an $A_-$-$A_+$ bimodule for any $\lambda \in k$. In
view of the PBW decomposition $H_z = B_- \otimes k[h] \otimes B_+$, it
will suffice to show that $h^i \in B_- M B_+\ \forall i$, for some
finite-dimensional $M$.

We claim that we may take $M = k \oplus kh \oplus \dots \oplus
kh^{2\deg(z)+1}$. Indeed,
\[ \lambda = t_z = ey^2 + hxy - fx^2 - \half hz - q_z \equiv z + \half hz
- q_z \mod B_- M B_+.\]

\noindent Now note that $\Delta = 4ef + (h^2 - 2h)$, whence (abusing
notation)
\[ h^a \Delta^b \in B_- \cdot k[h] / (h^{a+2b+1}) \cdot B_+\ \forall a,b
\geq 0. \]

\noindent In particular, $z, hz \in B_- M B_+$, so that $q_z \in B_- M
B_+$. On the other hand, since $\deg(q_z) = \deg(z) + 1$ and since
$h^{2\deg(z)+2} \in k q_z + k[f] M k[e]$, we get that $h^{2\deg(z)+2} \in
B_- M B_+$. From this, it follows that for any $i,\ h^i \in B_- M B_+$.
\end{proof}

It is an interesting problem to determine for which pairs of weights
$\lambda,\mu$, one has $I_\lambda := \Ann(V(\lambda)) \subset I_{\mu} :=
\Ann(V(\mu))$. As a first step, we have the following

\begin{theorem}
If the central element $t_z$ acts on $M(\lambda)$ by multiplication by
$\alpha$, then $\Ann(M(\lambda))$ is a two sided ideal generated by $t_z
- \alpha$ in $H_z$.
\end{theorem}

\begin{proof}
For the proof, at first we assume that $z=0$. In this case $t_z = t =
ey^2 + hxy - fx^2$ always acts by 0 on all Verma modules, so there is
only one block. Thus we need to show that $\Ann(M(\lambda)) = tH$. As
both sides of the desired equality are $\ad \mf{g}$-submodules of $H$,
and since the annihilator obviously contains $tH$, it will suffice to
prove that if we have any ($h$-weight vector) $g \in H$ such that $[f, g]
= 0 = g M(\lambda)$, then $g \in tH$. (We are considering ``lowest weight
vectors" inside $H$, which is a direct sum of finite-dimensional
$\mf{g}$-modules.)

Write $g$ as $\sum g_{ijl} h^l e^i x^j$ where $g_{ijl} \in k[f,y]$. Since
by assumption $[f,g] = 0$ then $g M(\lambda) = 0$ if and only if
\[ g v_\lambda = 0 = gy^n v_\lambda = \sum g_{ijl} [h^l, y^n] e^i x^j
v_\lambda = \sum g_{00l} [h^l, y^n] v_\lambda\ \forall n, \]

\noindent where the penultimate equality follows because $[x,y] = 0$ and
$[e,y] v_\lambda = 0$. But $h y^n = y^n h - y^n$, so we get
\[ g y^n v_\lambda = \sum_l g_{00l} y^n (h^l - (h - n)^l) v_\lambda = 0,
\]

\noindent whence (cancelling $y^n$ on the left in $M(\lambda) \cong B_- =
k[f,y]$, an integral domain) we get that $f(n) = 0$ for all $n$, where
\[ f(T) = \sum_{l>0} g_{00l} (\lambda - T)^l - \sum_{l>0} g_{00l}
\lambda^l \in k[T] \]

\noindent By Lemma \ref{Lpoly} (and induction on $l$), we conclude that
\begin{equation}\label{Enocomp}
g_{00l} = 0\ \forall l > 0
\end{equation}

Next, rewrite $g$ as $\sum_{n=0}^N \sum_{i=0}^n a_{in} x^{n-i} y^i$,
where $a_{in} \in \mf{Ug}$. Using the ``dividing trick" \eqref{Etrick},
we may assume that $g$ is not divisible by $y$ from the right, so some
$a_{0n} \neq 0$. Now, we have
\[ 0 = [f,g] = \sum_{i,n} (n-i) a_{in} x^{n-i-1} y^{i+1} + \sum_{i,n} [f,
a_{in}] x^{n-i} y^i, \]

\noindent so $[f, a_{i+1,n}] = (n-i) a_{in}$ for all $i,n$. In
particular, $[f,a_{0n}] = 0\ \forall n$. Since $H$ is a direct sum of
finite-dimensional $\mf{g}$-modules, wt$(g)$ must be nonpositive, so
wt$(a_{0n})\leq -n$.\medskip

There are only two steps remaining. First, we claim that $N>1$ if $g \neq
0$, and second, if so, then we can find $a \in tH$ such that $g-a$ has
``smaller $N$-value"; this finishes the proof, by induction on $N$.

Suppose $N=0$ first. Then by a result similar to Lemma \ref{L2}, $g =
a_{00} = p(\Delta) \cdot f^l$ for some $l \geq 0$ and $p \in k[T]$. If
this kills $y^n v_\lambda\ \forall \lambda$, then
\[ p \left( (\lambda - 2l - n)^2 + 2(\lambda - 2l - n) \right) = 0\
\forall n \]

\noindent and this would imply that $p$ is a constant, by Lemma
\ref{Lpoly}. This contradicts that $p \cdot f^l$ annihilates
$M(\lambda)$, unless $p=0$.

Next, suppose $N=1$ and $g = a_0 x + a_1 y + a_2$ (with all $a_i \in
\mf{Ug}$), so that $a_0, a_2 \in k[f, \Delta]$. (Then $a_2 = 0$ by
considering the parity of the possible weights.) Moreover, $a_1 = [e,
a_0] + b$, where $[f,b]=0$; therefore $a_1 y$ will contain a PBW monomial
not containing $e,x$ and containing $h$. But this contradicts equation
\eqref{Enocomp} above.\medskip

This proves the first step; moreover, $f | a_{0N}$, since wt$(a_{0N}) <
N$ and $a_{0N} \in k[f, \Delta]$. Now consider $a_{0N}/f$; as in the
proof of Proposition \ref{P3}, there exists an element $g' = (a_{0N}/f)
x^{N-2} + \sum_{i=1}^{N-2} c_i x^{N-2-i} y^i$ which commutes with $f$.
Thus, $g + g't \in \Ann(M(\lambda))$ commutes with $f$, and it is
divisible by $y$ from the right, hence we may divide by it. Proceeding by
induction on $N$, the result is proved when $z=0$.\medskip

Now let $z$ be arbitrary. Given $\lambda \in k$, recall the inclusion in
equation \eqref{Eann}: $\Gr(\Ann(M(\lambda))) \subseteq
\Ann(\Gr(M(\lambda))$. Moreover, $\Gr(M(\lambda))$ is just a Verma module
over $H$. Therefore if $g \in \Ann(M(\lambda))$, then $g = (t_z - \alpha)
g' + g''$, where $g''$ has lower filtration degree than $g$ (since $g''
\in \Ann(M(\lambda))$). Proceeding by induction on the filtration degree
of $g$, we are done.
\end{proof}\hfill

We conclude by considering the constant parameter case: $z=1$. The
following theorem describes the primitive spectrum of $H_z$, as well as
the multiplicities of irreducible modules in Verma modules.

\begin{theorem}
For $\lambda \neq \mu,\ V(\lambda),V(\mu)$ lie in the same block if and
only if $\lambda+\mu=-3$, and $M(\lambda)$ is irreducible if and only if
$\frac{3}{2}+\lambda$ is not a positive integer. Otherwise we have $0 \to
V(-3-\lambda) \to M(\lambda) \to V(\lambda) \to 0$, and $I_{-3-\lambda}
\subsetneq I_{\lambda}$.
\end{theorem}

\noindent In particular, (primitive) annihilator ideals for $\lambda \neq
\mu$ are either not comparable ($\lambda \neq -\mu-3$), or equal
($\lambda = -\mu - 3 \notin \half + \Z$), or strictly comparable
(otherwise).

\begin{proof}
Recall that in this case, the central element is equal to $t_1 := ey^2 +
hxy - fx^2 - \half h + \half \Delta$, so it acts on $V(\lambda)$ by the
scalar $1 + \half (\lambda+((\lambda+1)^2-1))$, hence $V(\lambda),V(\mu)$
lie in the same block if and only if $\lambda = \mu$, or $\lambda +
\mu=-3$.

Next, note that $[x, y^2 + 2f] = 2y - 2y = 0$, therefore
$x(y^2+2f)^nv_{\lambda}=0$. We now determine when $(y^2+2f)^n
v_{\lambda}$ is annihilated by $e$. Using that $[x, y^2 + 2f] = 0$, we
have
\begin{eqnarray*}
0 & = & e(y^2 + 2f)^n v_{\lambda} = [e, (y^2 + 2f)^n] v_{\lambda}\\
& =& \sum_{l<n} (y^2 + 2f)^l (2(yx+h) + 1) (y^2 + 2f)^{n-l-1}\\
& = & n (2\lambda + 3 - 2n) (y^2 + 2f)^{n-1} v_{\lambda}.
\end{eqnarray*}

\noindent Hence if $n$ is minimal among those for which $(y^2 + 2f)^n
v_{\lambda}$ is a maximal vector, we must have $\lambda = n -
\frac{3}{2}$. Now assume that $g=\sum_i a_i f^i y^{n-2i} v_{\lambda}$ is
a maximal vector; then $x \cdot g$ must vanish. In other words,
\[ 0 = \sum a_i [x, f^i y^{n-2i}] v_{\lambda} = \sum (n-2i) a_i f^i
y^{n-2i-1} v_{\lambda} - \sum i a_i f^{i-1} y^{n-2i+1} v_{\lambda}. \]

\noindent This implies that $(n-2i)a_i = (i+1)a_{i+1}$ for all $i.$
Hence, $n$ is even and this system of equalities has exactly one solution
up to multiplication by a constant; therefore $g=(y^2+2f)^{n/2}
v_\lambda$.

To conclude, we have shown that $M(\lambda)$ is irreducible if
$\frac{3}{2}+\lambda$ is not a positive integer, and otherwise we have
the desired short exact sequence. Finally, since $(y^2+2f)^n\in
\Ann(V(\lambda))$, therefore
\[ \Ann(V(\mu)) = \Ann(M(\mu)) = \Ann(M(\lambda)) \subsetneq
\Ann(V(\lambda)), \]

\noindent where $\lambda + \mu = -3$.
\end{proof}\medskip

\begin{acknowledgement}
We thank the referee for his patient reading, and for making numerous
suggestions and comments, which helped in improving this manuscript.
\end{acknowledgement}

\end{document}